\documentclass[EJP]{ejpecp} 

\SHORTTITLE{DLAS and the increasing convex order}

\TITLE{Diffusion-limited annihilating systems and the increasing convex order\thanks{RB was partially supported by NSF grant DMS-1855516 and the 2020 Baruch College Discrete Math REU. PB was partially supported by NSF grant DMS-1855516. TJ was partially supported by NSF grant DMS-1811952 and PSC-CUNY Award \#62628-00 50. MJ was partially supported by NSF grant DMS-1855516.}} 

\AUTHORS{%
  Riti Bahl\footnote{Emory University,
    \EMAIL{riti.bahl@emory.edu}}
  \and 
  Philip Barnet\footnote{Bard College,
    \EMAIL{pb7734@bard.edu}}
   \and
   Tobias Johnson\footnote{College of Staten Island,
    \EMAIL{Tobias.Johnsun@csi.cuny.edu}}
    \and 
   Matthew Junge\footnote{Baruch College,
    \EMAIL{Matthew.Junge@baruch.cuny.edu}}
  }


\AMSSUBJ{60K35, 60J80, 60J10} 

\SUBMITTED{May 11, 2021} 
\ACCEPTED{June 10, 2022} 

\VOLUME{27}
\YEAR{2022}
\PAPERNUM{84}
\DOI{10.1214/22-EJP808}

\ABSTRACT{We consider diffusion-limited annihilating systems with mobile $A$-particles and stationary $B$-particles placed throughout a graph. Mutual annihilation occurs whenever an $A$-particle meets a $B$-particle. Such systems, when ran in discrete time, are also referred to as parking processes. We show for a broad family of graphs and random walk kernels that augmenting either the size or variability of the initial placements of particles increases the total occupation time by $A$-particles of a given subset of the graph. A corollary is that the same phenomenon occurs with the total lifespan of all particles in internal diffusion-limited aggregation.}

\usepackage{amssymb}
\usepackage[shortlabels]{enumitem}
\usepackage{tikz}
\usepackage{placeins}
\usetikzlibrary{decorations.pathmorphing}
\usepackage{subfigure, float}
\usepackage{comment}
\usepackage{hyperref}
\usepackage{mathtools}
\mathtoolsset{showonlyrefs}

\newcommand{\HOX}[1]{\marginpar{\scriptsize #1}}

\newcommand{\ind}[1]{\mathbf{1}{\{ #1 \}}}
\newcommand{\1}{\mathbf{1}}
\newcommand{\icx}{\preceq_{icx}}
\newcommand{\pmicx}{\preceq_{icx}}

\newcommand{\0}{x}
\renewcommand{\rho}{x}
\renewcommand{\phi}{\varphi}
\newcommand{\E}{\mathbf E}
\renewcommand{\P}{\mathbf P}
\newcommand{\eqd}{\overset{d}=}

\DeclareMathOperator{\var}{var}

\usepackage{theoremref}

\begin{document}

\section{Introduction}

We study a class of diffusion-limited annihilating systems (DLAS) in which $A$-particles diffuse across a graph interspersed with stationary $B$-particles. Mutual annihilation $A+B \to \varnothing$ occurs whenever opposite particle types meet. 
 DLAS were introduced by physicists as toy models exhibiting anomalous kinetic behavior observed in more complicated reactions. The overarching finding was that spatial concentration fluctuations of reactants play a significant role in solute decay \cite{gopich1996kinetics, tartakovsky2012effect}. Such fluctuations arise naturally in physical systems with thermal fluctuations \cite{ovchinnikov1978role, toussaint1983particle}, turbulent
flows \cite{hill1976homogeneous}, and porous media \cite{raje2000experimental}.
 
For two-type systems, Bramson and Lebowitz gave an extensive rigorous analysis of the setting in which both particle types are mobile and diffuse at the same rate \cite{BL2, BL3, BL4, BL5}. Cabezas, Rolla, and Sidoravicius made progress on two-type systems with asymmetric diffusion rates \cite{CRS}.
More results concerning the limiting density of particles were obtained in \cite{johnson2020particle, dlas_star, parking_on_integers}. 

The extreme case, in which $B$-particles are stationary, has also been studied by combinatorialists and probabilists under the name \emph{parking} \cite{KW}. 
This comes from viewing $A$-particles as cars in search of $B$-particle spots. 
The impetus of our present work comes from results for parking on Galton-Watson trees \cite{tree, curien2019phase, contat2020sharpness, bahl2019parking, chen2021}. 
These articles considered the setting with one $B$-particle per site along with an independent and identically distributed number of $A$-particles (any initially overlapping $A$- and $B$-particles are cancelled out). 
The graph is directed so that all $A$-particles move towards the root in discrete time.
The total number of $A$-particles to visit the root has two phases: \emph{transience} (finitely many visits almost surely) and \emph{recurrence} (infinitely many visits with positive probability). Another version of a DLAS with stationary $B$-particles was introduced in \cite{rivera2019dispersion}, which studied simultaneous internal  diffusion-limited aggregation on finite graphs. This model corresponds to a DLAS with a finite number of $A$-particles at the root of a connected graph and one $B$-particle at each nonroot vertex.

An interesting feature of parking on directed Galton-Watson trees is that  the phase behavior depends on more than the average initial particle density.
Curien and H\'enard gave a precise characterization of the transience/recurrence phase behavior on critical Galton-Watson trees that involves the mean and variance of the number of $A$-particles as well as the variance of the offspring distribution for the tree \cite{curien2019phase}. The phase transition was later proven to be sharp by Contat \cite{contat2020sharpness}.   

There are no known explicit criteria for recurrence and transience for parking on supercritical Galton-Watson trees. However, Collett, Eckmann, Glaser, and Martin gave a precise value of the phase transition in terms of exponential moments for the process on binary trees with $A$-particles started exclusively from the leaves \cite{collet1983study}. They were interested in these dynamics because of a connection to a spin glass model. 
For the usual parking process Bahl, Barnet, and Junge demonstrated through an example on $d$-ary trees that the phase state depends on more than just the mean initial density of particles \cite[Proposition 7]{bahl2019parking}. 
In an attempt to qualitatively describe this phenomenon \cite[Theorem 8]{bahl2019parking} further showed that the total number of visits to the root increases when the underlying placement of $A$-particles is made more volatile. An additional definition is needed to describe the result.

The \emph{increasing convex order} (icx order) is a less commonly used stochastic ordering that rewards
random variables for being larger or more volatile.
We say that $X \icx Y$ for two random variables $X$ and $Y$ supported on the real numbers $\mathbb R$, if for all increasing convex functions $\phi\colon \mathbb R \to \mathbb R$ it holds that $\E \phi(X) \leq \E \phi(Y)$ provided the expectations exist.  
See \cite{SS} for a thorough discussion of the icx and related stochastic orders. 

Note that the icx order is weaker than the \emph{standard order} $X \preceq_{sd} Y$, which is defined
as $\P(X \geq a) \leq \P(Y \geq a)$ for all $a$. An equivalent definition that more closely resembles the definition of the icx order is that $X \preceq_{sd} Y$ if and only if $\E \phi(X) \leq \E \phi(Y)$ for all increasing functions $\phi$ for which the expecations exist \cite[(1.A.7)]{SS}. Another familiar definition of the standard order \cite[Theorem 1.A.1]{SS} is that there exist two random variables $\hat X, \hat Y$ defined on the same probability space with $\hat X \overset{d} = X$ and $\hat Y \overset{d} = Y$ and $\P(X\leq Y) =1$. (Here $\overset d =$ denotes distributional equality.) The icx order has an analogous formulation \cite[Theorem 4.A.5]{SS} with the modification that $\{\hat X, \hat Y\}$ is a submartingale i.e., $E[\hat Y \mid \hat X] \geq \hat X$.

As an example, suppose that $X=c$ with probability one for some constant $c \in \mathbb R$, while $Y$ is any other random variable with mean $c$.
Then $X \icx Y$ by Jensen's inequality. However, $X$ and $Y$ are not comparable in the standard order. Here are a few more examples to keep in mind.
\begin{example} \thlabel{ex:1}
Let $X$ and $Y$ be real-valued random variables with $X \icx Y$. 
\begin{enumerate}[label = (\roman*)]
    \item Since $x \mapsto x$ is convex and increasing, we have $\E X \leq \E Y$.
    \item If $X$ and $Y$ are nonnegative and $\E X = \E Y$, then the fact that $x \mapsto x^2$ is convex and increasing implies that $\var(X)= \E X^2 - (\E X)^2 \leq \E Y^2 -( \E Y)^2 =  \var(Y)$.
    \item For $t \in [0,1]$ the function $x \mapsto t^x$ is convex and decreasing. If $X$ and $Y$ are nonnegative, then $\E t^X \geq \E t^Y$. Taking the limit as $t \downarrow 0$ gives $\P(X=0) \geq \P(Y =0)$.
    \item  Let $(X_T)_{T\geq 0}$ and $(Y_T)_{T \geq 0}$ be collections of nonnegative random variables supported on $[0,\infty)$ with pointwise limits $X = \lim_{T \to \infty} X_T$ and $Y = \lim_{T \to \infty} Y_T$ that exist almost surely. Suppose further that $X_T \icx Y_T$ for all $T \geq 0$  As in (iii), we have $\E t^{X_T} \geq \E t^{Y_T}$ for all $t \in [0,1]$. It follows that
    $$\P(X < \infty) = \lim_{t \uparrow 1} \lim_{T \to \infty}  \E t^{X_T} \geq \lim_{t \uparrow 1} \lim_{T \to \infty}  \E t^{Y_T} = \P(Y < \infty).$$ 
    Equivalently, $\P(X = \infty) \leq \P(Y = \infty)$. 
\end{enumerate}
\end{example}

Let $V$ (and $V'$) be the number of $A$-particles that arrive to the root for parking on a directed Galton-Watson tree with an i.i.d.\ $\eta$-distributed (i.i.d.\ $\eta'$-distributed)  $A$-particles initially at each site, one of which immediately cancels with the one $B$-particle also initially at the site. 
\cite[Theorem 8]{bahl2019parking} proved that if $\eta \icx \eta'$, then $V \icx V'$. 
Combining this with \thref{ex:1} (iv) implies that recurrence is preserved after increasing $\eta$ in the icx order. 

There are a few other results concerning particle systems and nonstandard stochastic orders. 
Johnson and Junge proved that the frog model, with $A+B \to 2A$ reactions so that mobile particles activate stationary particles, exhibits the opposite phenomena of parking: more volatility reduces root visits \cite{JJ}. Johnson and Rolla later proved via an example on regular trees that transience and recurrence of that model depends on more than the average density of the initial configuration \cite{johnson2019}. 
A similar relationship between concentration and expansion occurs with the limiting size of the ball in first passage percolation \cite{BK, Marchand}. 
Recently, Hutchcroft studied the effect of a weaker stochastic order called the germ order on transience and recurrence sets for branching random walk \cite{hutchcroft2020transience}.

\subsection{Statement of results}

 We extend \cite[Theorem 8]{bahl2019parking} to general graphs, initial particle configurations, and paths. 
 We begin with a more formal construction that follows the notation from \cite{CRS}.
 Fix a locally finite graph $G$. We write $x \in G$ and $H \subseteq G$ to denote that $x$ is an element and $H$ is a subset of the vertex set of $G$. The systems we consider here are described by time-indexed counts at each vertex $\xi = (\xi_t(x))_{t \geq 0, x \in G}.$ When $\xi_t(x)>0,$ it denotes the number of $A$-particles at $x$ at time $t$. If $\xi_t(x) <0$, then there are $|\xi_x(t)|$ many $B$-particles at $x$ at time $t$.  
 
 We call $\xi_0 = (\xi_0(x))_{x \in G}$ the \emph{initial conditions}. 
For each $x \in G$ and $j \in \mathbb Z^+$ let $S^{x,j} = (S_t^{x,j})_{t \geq 0} \colon [0,\infty) \to G$ be a right-continuous path with left limits started at $x$. 
We assign to the $j$th particle counted by $\xi_0(x)$ the path $S^{x,j}$.
We further assign to the $j$th $A$-particle counted by $\xi_0(x)$ a \emph{braveness} $h^{x,j} \in [0,1]$
and to the $j$th $B$-particle braveness $h^{x,-j}\in[0,1]$, to be used to break ties
in deciding which particles to annihilate when multiple $A$- and $B$-particles end up
at the same site. We assume that the bravenesses $(h^{x,j})_{x \in G, j \in \mathbb Z \setminus \{0\}}$ are distinct.
Since $B$-particles do not have paths or anything else to distinguish them, there is no real need to
determine which $B$-particle on a site is annihilated, but it will be convenient to do so
when we consider variant processes. 

We call $(S,h) = (S_t^{x,j}, h^{x,j})_{x \in G, j \in \mathbb Z \setminus \{0\}, t \geq 0}$ the \emph{instructions}.
Particles follow their assigned trajectories and when one or more $A$-particles arrive to a site containing $B$-particles, the bravest of the $A$-particles mutually annihilate with $B$-particles until there are no remaining pairs of opposite type particles at the site. 
We say that $(\xi_0,S,h)$ is \emph{regular} if all of the following hold:
\begin{itemize}
    \item The initial conditions $(\xi_0(x))$ are independent over $x\in G$; the instructions
      $(S^{x,j})$ and $(h^{x,j})$ are both independent over $x\in G$ and $j\in\mathbb{Z}\setminus \{0\}$; 
      and the three collections of random variables are independent of each other as well.
    \item For $j >0$, the $S^{x,j}$ are random walk paths with the same transition kernel, either
      entirely in continuous time (i.e., jumping at times given by a unit intensity Poisson process) or entirely in 
      discrete time (i.e., jumping at positive integer times).
    \item For $j<0$, the paths $S_t^{x,j} \equiv x$ so that $B$-particles are stationary.
    \item Each $h^{x,j}$ is a uniform random variable on $[0,1]$.
    \item The instructions and initial conditions are such that the system
      is well defined and can be locally approximated by a system with finitely many particles,
      in the sense of \thref{rmk:finite.approximation}.

\end{itemize}

\begin{remark}\thlabel{rmk:finite.approximation}
   We say that $(\xi_0,S,h)$ can be locally approximated by a system with finitely many particles
   if the following holds:
   Let $G_1\subseteq G_2\subseteq \cdots\subseteq G$ be any sequence of finite
   subgraphs whose union is $G$. Let $H \subseteq G$ be a finite collection of vertices and let $\bigl(\xi_t^{(n)}(x)\bigr)_{t\geq 0,x\in G}$ be DLAS defined
   by $(S,h)$ with initial conditions $\xi_0(x)\ind{x\in G_n}$.
   Then the sequence $(\xi_t^{(n)}(x))_{0\leq t\leq T,x\in H}$ indexed by $n$ is almost surely
   eventually constant, and its limit does does not depend on the choice of
   $G_1,G_2,\ldots$.

   The point of this rather technical condition is simply to avoid pathological examples.
   A system defined by $(\xi_0,S,h)$ may not yield a well-defined DLAS $(\xi_t(x))_{t\geq 0,x\in G}$ if,
   for example, blow-up occurs and infinitely many particles move to a single vertex in a finite amount
   of time. It is proven in \cite[Appendix~A]{CRS} that local approximation holds when 
   the initial conditions satisfy $\sup_{x\in G}\E\xi_0(x)<\infty$ and the paths are
   random walks on any graph with a transitive unimodular group of automorphisms that
   commutes with the probability kernel of the random walk. It is not hard to show that
   it holds for a DLAS such that a system of noninteracting random walks with the same initial 
   conditions has finitely many visits in finite time to any site.
\end{remark}

Let $H \subseteq G$. Define the \emph{occupation time of $H$ by $A$-particles up to time $T\geq 0$} as
\begin{align}V_T(H,\xi_0, S,h) = V_T = \sum_{x \in H} \int_0^T\ind{\xi_t(x)>0} \xi_t(x)dt
\label{eq:V}.\end{align}
We set $V_T = \infty$ whenever $\sum_{x \in H} \xi_0(x) \ind{\xi_0(x) >0} = \infty$. To avoid this pathology we assume that either $H$ or the number of $A$-particles in the system is finite:
\begin{align}
    \min \left\{|H|, \textstyle \sum_{x \in G} \xi_0^+(x)\right \} < \infty \text{ almost surely}\label{eq:fin}.
\end{align}
The prototypical choice is $H = \{x\}$, so that $V_T$ measures the occupancy time of a single site. Another interesting choice is $H=G$ so that $V_T$ equals the aggregate time all of the $A$-particles are in motion. This case is only meaningful when \eqref{eq:fin} holds. We further remark that \eqref{eq:fin} along with the requirement $\sup_{x \in G, t \geq 0} \E |\xi_t(x)|< \infty$ ensure that $V_T$ is almost surely finite for all $T>0$. 

Given another DLAS $\xi'$ define $V'_T$ to be the analogue of \eqref{eq:V} for $\xi'$. 
A useful condition for comparing two systems $(\xi_0, S,h)$ and $(\xi_0', S',h')$ is
\begin{align}
    \text{$(\xi_0, S,h)$ and $(\xi_0',S',h')$ are regular with $S\overset d = S'$ and $h \overset d = h'$} \label{eq:comparable}.
\end{align}
We write $\xi_0 \pmicx \xi_0'$ if $\xi_0(x) \icx \xi_0'(x)$ for all $x \in G$. 

\begin{theorem}\thlabel{thm:main}
Fix $H \subseteq G$. Suppose that \eqref{eq:fin} and \eqref{eq:comparable} hold. If  $\xi_0 \pmicx \xi_0'$, then $V_T \icx V_T'$ for all $T \geq 0$.
\end{theorem}

\begin{proof}
The result follows when $\sum_{x \in G} |\xi_0(x)|<\infty$ and $\sum_{x \in G} |\xi_0'(x)| <\infty$ from \thref{prop:i_ii} and \thref{prop:icx}. If there are infinitely many $A$-particles present, then
the finite approximation assumption in the definition of regularity
ensures that $V_T$ is realized as the limit as $R \to \infty$ of the systems that only contain the particles from $\xi_0$ that lie in the balls $\cup_{x \in H} \mathbb B(x, R)$. Taking such a limit gives the statement for $V_T$ when the initial configuration contains infinitely many particles.
\end{proof}

Our proof comes in two parts. First we show that $V_T$ is an \emph{icx statistic} (see
\thref{def:icx.stat,prop:i_ii}), and then we show a generalized version of \thref{thm:main}
that holds for all icx statistics. This opens the door to establishing that other statistics
besides $V_T$ respect the icx order; see Section~\ref{sec:further} for more discussion.

The analogue of \thref{thm:main} for the standard order holds immediately by monotonicity
of the process with respect to its initial conditions \cite[Lemma~3]{CRS}.
Because our result uses the weaker icx order, it applies more broadly. In particular,
two different distributions with the same mean can be comparable in the icx order but are never
comparable in the standard order.

To illustrate our theorem and begin our discussion of why the icx order is natural to consider,
we give an example where $\xi_0\icx \xi'_0\icx\xi''_0$ and the conclusion $V_T\icx V_T'\icx V_T''$
can be seen by direct calculation.

\begin{example} \thlabel{ex:icx} Set $G=\mathbb Z$ and $H=0$. Let $S$ consist of discrete symmetric nearest neighbor random walk paths. Now, we define systems in this environment with three different
sets of initial conditions. The DLAS $\xi$ starts with one $B$-particle at position~$1$
and one $A$-particle at position~$2$. For $\xi'$, we place either zero or two $B$-particles at position~$1$
with equal probability
and one $A$-particle at position~$2$.
And $\xi''$ begins with $0$ or $2$ $B$-particles at $1$ with equal probability and $0$ or $2$
$A$ particles at position~$1$ with equal probability. In all three systems, there are no particles
initially outside of positions~$1$ and $2$. To summarize,
\begin{itemize}
  \item $\xi_0(1)=-1$, $\xi_0(2)=1$
  \item $\xi_0'(1)= \begin{cases}0&\text{with probability~$1/2$,}\\-2&\text{with probability~$1/2$,}\end{cases}$ and $\xi'_0(2)=1$, 
  \item $\xi''(1) = \begin{cases}0&\text{with probability~$1/2$,}\\-2&\text{with probability~$1/2$,}\end{cases}$ and $\xi_0''(2)= \begin{cases}0&\text{with probability~$1/2$,}\\2&\text{with probability~$1/2$;}\end{cases}$
  \item and $\xi_0(x)=\xi'_0(x)=\xi''_0(x)=0$ for $x\notin\{1,2\}$.
\end{itemize}

We have $\E \xi_0(x) = \E \xi_0'(x)=\E\xi_0''(x)$ for all $ x \in G$, and Jensen's inequality confirms the intuitively obvious fact that $\xi_0$, $\xi'_0$, and $\xi''_0$ are increasingly volatile, i.e., $\xi_0 \icx \xi_0'\icx\xi_0''$. Observe that $V_T=0$ a.s., while
\begin{align*}
  V_T'&=\begin{cases}0&\text{with probability~$1/2$,}\\L_T&\text{with probability~$1/2$,}\end{cases}\\
    \intertext{and}
  V_T''&=\begin{cases}0&\text{with probability~$3/4$,}\\L_T+L_T'&\text{with probability~$1/4$,}\end{cases}
\end{align*}
where $L_T$ and $L'_T$ are the local times at $0$ up to time~$T$ of two independent random walks started at
$2$. We leave it as an exercise to compute directly that $V_T\icx V_T'\icx V_T''$ (we note
that no knowledge of the distribution of $L_T$ is needed).
We also observe that $\E V_T < \E V_T' = \E V_T''$, demonstrating that the expectation of $V_T$
can but need not strictly increase when the initial conditions increase in the icx order.
\end{example}
The reason that $V_T$ is increasing with respect to changes in the initial configuration in the icx
order is because of a convexity property inherent to the model.
Essentially, the gain to $V_T$ when adding two $A$-particles to the system
is more than the sum of gains from each $A$-particle alone.
As we can see in the previous example, this is because one $A$-particle can help
the other by eliminating obstacles on its behalf.
More discussion is presented at the start of Section~\ref{sec:overview}.

\subsection{Applications}
\thref{thm:main} applies to the systems considered in \cite{curien2019phase, parking, dlas_star, johnson2020particle, bahl2019parking, CRS,rivera2019dispersion}. This includes a broad class of transitive unimodular graphs and Galton-Watson trees. Such graphs are good to keep in mind, but we note that significantly more general graphs are also covered since we require minimal regularity.

One application of \thref{thm:main} that follows from \thref{ex:1} (i), (iii), and (iv) is that increasing the initial condition in the icx order has the following effects.

\begin{corollary} \thlabel{cor:visits}
 Fix $H \subseteq G$. If \eqref{eq:fin} and \eqref{eq:comparable} hold and  $\xi_0 \pmicx \xi_0'$ then:
 \begin{enumerate}[label = (\roman*)]
     \item $\E V_T \leq \E V'_T$ for all $T \geq 0$, 
     \item $\P(V_\infty = \infty) \leq \P(V'_\infty= \infty)$, and
     \item $\P(V_T = 0) \geq \P(V_T' =0)$ for all $T \geq 0$.
\end{enumerate} 
\end{corollary}

%
To make \thref{thm:main} more concrete we derive as a consequence that taking $\xi_0(x)$ to be independent $\pm 1$-valued random variables minimizes $V_T$ in the icx order across all other initial configurations with the same mean number of particles at each site. Note that such initial configurations are used in \cite{parking,parking_on_integers, johnson2020particle, tree, bahl2019parking}.

\begin{corollary}\thlabel{cor:concrete}

Fix a graph $G$, a finite subset $H\subseteq G$, and possibly distinct values $p_x \in [0,1]$ for each $x \in G$. Let $U(x)$ be independent random variables uniformly distributed on $(0,1)$. Let $\alpha(x)$ and $\beta(x)$ be independent and identically distributed random variables on the nonnegative integers with mean 1.
For all $x \in G$ define 
\begin{align} 
    \xi_0(x) &= \ind{U(x) \leq p_x} - \ind{U(x) > p_x}\\
    \xi_0'(x)&= \ind{U(x)\leq p_x} \alpha(x) - \ind{U(x) > p_x} \beta(x)
\end{align}
Suppose that $(\xi_0, S,h)$ and $(\xi_0', S, h)$ are regular and let $V$ and $V'$ be the total occupation time of $H$ for the two systems, respectively. It holds that $V \icx V'$. 

\end{corollary}
\begin{proof}
  The results \cite[Proposition 15 (b)]{JJ} and \cite[Theorem 4.2.A]{SS} are easily adapted to prove that 
  \begin{align} 
  \ind{U(x) \leq p_x} &\icx \ind{U'(x) \leq p_x} \alpha'(x)\\
  -\ind{U(x) > p_x} &\icx - \ind{U'(x) >p_x} \beta'(x).
  \end{align}
Summing both sides and applying closure under mixtures of the icx order \cite[Theorem 4.A.8 (b)]{SS} gives that $\xi_0(x) \icx \xi_0'(x).$ \thref{thm:main} implies $V \icx V'$. 
\end{proof}

As mentioned below \eqref{eq:V}, the case $H=G$ gives the aggregate time $A$-particles are in motion. Suppose further that $G$ is a rooted graph with all non-root sites initially containing one $B$-particle and $n\geq0$ $A$-particles at the root. Using coupling methods different from ours, the quantity $V_T$ in this setup, referred to as $W$ in \cite{rivera2019dispersion}, was shown in  \cite[Lemma 4.4]{rivera2019dispersion} to have the same distribution as $L(n)$ the combined path length of the particles
in $n$-iterations of the usual, sequential version of internal diffusion-limited aggregation \cite{lawler1992internal, jerison2012logarithmic}. The authors of \cite{rivera2019dispersion} noted that this equivalence ``motivates the study of $W$ for general graphs...''
In this line, we obtain the following consequence of \thref{thm:main}.


\begin{corollary} \thlabel{cor:IDLA}
Let $L(n)$ be the combined lifespans of $n$ particles executing internal diffusion-limited
aggregation from the root of a given graph. Let $\eta$ and $\eta'$ be nonnegative random variables. If $\eta\icx \eta'$, then $L(\eta)\icx L(\eta')$.
\end{corollary}




\subsection{Proof overview} \label{sec:overview}

The first and main part of the proof is \thref{prop:i_ii}, where we show $V_T$ is an \emph{icx statistic} 
(see \thref{def:icx.stat}). The essence of the definition is that the statistic increases with the addition of a single $A$-particle (condition (d)), and that the increase from adding two $A$-particles at once is greater than the increase from adding two $A$-particles individually in two separate systems (condition (e)). Our proof amounts to confirming the intuition that one of the extra particles may clear space for the other by annihilating with a $B$-particle that would destroy each particle separately in individual augmented systems.
This property helps explain why greater volatility in the initial configuration causes an icx statistic
to increase. After all, if the gain from adding two particles is more than twice that of adding one,
we would rather have two particles or none at a site with probability $1/2$ than one particle
with probability $1$ as illustrated in \thref{ex:icx}.

To establish these properties formally, we introduce couplings of systems with  extra $A$-particles.  
Our approach is inspired by the \emph{tracer} construction from \cite{CRS} but differs slightly. In light of conditions \ref{i:Phi.1st.order} and \ref{i:Phi.2nd.order} from \thref{def:icx.stat}, we need a way to compare what happens in systems with extra $A$-particles. 
If we simply add two $A$-particles to the system and let all other particles proceed with their original
instructions, we obtain a coupling
satisfying \ref{i:Phi.1st.order}, the monotonicity requirement, but not
\ref{i:Phi.2nd.order}, the convexity requirement. 
To make the convexity requirement hold, we must prioritize one of the added
$A$-particles over the other, which is where our paper diverges from \cite{CRS}. 
For example, it may happen that one of the extra $A$-particles, call it $X$, destroys a $B$-particle at $x$, thus allowing the other extra $A$-particle, call it $Y$, to survive a later visit to $x$. When $Y$ arrives to $x$ the law of the DLAS is preserved whether we allow $Y$ to continue its assigned path, or if we instead have it pick up and extend the path of $X$. 
We call this process either the \emph{tracer system} or the \emph{flipped tracer system} depending
on which of the two extra $A$-particles we prioritize.
We will leave the details to the next section, but from the tracer system we track the effect
of adding $Y$ after $X$, while from the flipped tracer system we track the effect of adding
$Y$ alone. We then prove in \thref{lem:difference} that $Y$ traverses more of its preassigned
path if it is added after $X$, thus establishing \ref{i:Phi.2nd.order}. 
We note that this comparison fails when $B$-particles move,
as we discuss in \thref{rem:B}. 

The second part of the proof is to show that icx statistics are monotone under changes
to the initial conditions in the icx order. 
We do this under the assumption that the system has only finitely many particles
(\thref{prop:icx}). From there it is straightforward to extend the result to infinite systems.

The overall structure of the proof is similar to the one used in \cite{JJ} to prove a comparison result for the frog model.
Some differences are that: (a) here we consider a two-type particle system, 
(b) we use the icx order rather than the increasing concave (icv) order, and 
(c) the definition of icx statistic is more complicated than the analogous definition in \cite{JJ}.

To address difference~(a), 
we make use of the basic equivalence between removing $B$-particles as adding $A$-particles. So, it is sufficient to restrict our focus to the impact of adding $A$-particles as discussed in the previous paragraph. This makes for no technical difficulty.
Difference~(b) is a consequence of differences between DLAS and the frog model. Superadditive statistics
(see \thref{def:icx.stat}) behave well for DLAS as opposed to subadditive ones for the frog
model. But this difference barely affects the proofs because the icx and icv orders are interchangeable via \cite[Theorem 3.A.1]{SS}.

Difference~(c) is more substantial. 
For all icv statistics considered in \cite{JJ}, it is easy to prove that they are so. In contrast,
establishing that $V_T$ is an icx statistic is most of the work of this paper. The difference is
that statistics in the frog model behave subadditively when adding particles to the system in
the most straightforward way. On the other hand, the statistics considered in this paper
behave superadditively only when particles are added via the complicated coupling described
in Section~\ref{sec:couplings}.

We further remark that \cite{JJ} considers a weaker stochastic order, the \emph{probability generating function order} (pgf order), alongside the icv order. We say that $X$ is dominated by $Y$ in the pgf order if the probability generating function $\E t^X$ of $X$ is pointwise at least as large as that of $Y$ for $t \in [0,1]$.
With more technical difficulty, it is proven in \cite{JJ} that the number of visits to the root in the frog model also respects this stochastic order. 
The analogous generalization to the convex version of the pgf order for \thref{thm:main} would be to prove that $V_T$ respects the stochastic order
defined by $\E\varphi(X)\leq\E\varphi(Y)$ holding for all smooth functions $\varphi$ with all
derivatives positive.
It is quite possible that this is true and could be proven by the methods of this paper,
but we have not attempted it since there is no obvious application of the result.

\section{Icx statistics and tracer couplings} \label{sec:couplings}
We begin with a definition that we will relate to the icx order in the next section. Throughout this section we assume that any DLAS under discussion has finitely many $A$- and $B$-particles in the initial configuration. 

\begin{definition}\thlabel{def:icx.stat}
  Let $(\xi_0,S,h)$ be regular with $\sum_{x \in G} |\xi_0(x)| < \infty$ a.s. Let $\xi_{0,k}$ be the same as $\xi_0$ except that
  $\xi_{0,k}(x)=k$ for some given $x\in G$.
  We call a functional $f$ an \emph{icx statistic} if for all $x\in G$ and $k\in \mathbb Z$
  there exists a coupling $(\Phi,\Phi^X,\Phi^Y,\Phi^{X,Y})$ such that
  \begin{enumerate}[(a)]
    \item $\Phi\eqd f(\xi_{0,k},S,h)$\label{i:Phi}
    \item $\Phi^X\eqd\Phi^Y\eqd f(\xi_{0,k+1},S,h)$\label{i:Phi.X}
    \item $\Phi^{X,Y}\eqd f(\xi_{0,k+2},S,h)$\label{i:Phi.X.Y}
    \item $\Phi^X\geq\Phi$ a.s.\ and $\Phi^Y\geq\Phi$ a.s.\label{i:Phi.1st.order}
    \item $\Phi^{X,Y}-\Phi^X-\Phi^Y+\Phi\geq 0$ a.s. \label{i:Phi.2nd.order}
  \end{enumerate}
\end{definition}

The purpose of this section is to prove the following:
\begin{proposition} \thlabel{prop:i_ii}
  $V_T$ defined at \eqref{eq:V} is an icx statistic.
  
\end{proposition}

\subsection{The tracer system and the flipped tracer system } \label{def:tracer}

We now introduce the coupling we use to track the effect of adding extra particles
at a given vertex $x$.
We start with a regular system $(\xi_0,S,h)$. Let $k=\max(\xi_0(x)+1, 1)$, and let
$X=S^{x,k}$ and $Y=S^{x,k+1}$.
We will describe two modified versions of the underlying DLAS
where we add particles with paths $X$ and $Y$ with special behavior.

  In both systems,
  start with the initial conditions given by $\xi_0$ together with two special particles
  at a given site~$x$ that we call the $X$-tracer and the $Y$-tracer. 
  Each tracer has two possible states, $A$ and $B$.
  The nontracer particles follow their assigned instructions from $(S,h)$.
  The $X$- and $Y$-tracers follow paths $X$ and $Y$, respectively, when they are in state~$A$. 
  When a tracer enters state~$B$, it pauses and remains stationary until it enters state~$A$ again,
  at which point it continues following its path $X$ or $Y$ starting from when it was paused.
  
  In the \emph{tracer system}, we assign the $X$-tracer braveness $-2$ and the $Y$-tracer braveness $-1$; that is,
  the $Y$-tracer has lower braveness than all other particles except for the $X$-tracer, which has
  the lowest braveness. In the \emph{flipped tracer system}, we assign the $X$-tracer braveness $-1$ and
  the $Y$-tracer braveness $-2$. The tracer system gives the $X$-tracer priority to be in state $A$, while the flipped tracer system prioritizes the $Y$-tracer. In both systems tracers start out in state $A$.

  In the following interaction rules for the system's evolution,
  we treat a tracer as if were an $A$-particle when it is in state~$A$ and as a $B$-particle when it is
  in state~$B$. If $A$-particles and $B$-particles finds themselves together on a site, then the bravest
  $A$-particle and the bravest $B$-particle interact as follows:
  \begin{enumerate}[(a)]
    \item If neither particle is a tracer, then they interact as usual by mutual annihilation.
      \label{i:nontracers}
    \item If the $A$-particle is a tracer and the $B$-particle is not, then the $B$-particle annihilates
      and the tracer particle switches to state~$B$.
    \item If the $B$-particle is a tracer and the $A$-particle is not, then the $A$-particle annihilates
      and the tracer particle switches to state~$A$.
    \item If both particles are tracers, then the $Y$-tracer takes priority to be in state~$A$ in the
      tracer system, while the $X$-tracer takes priority to be in state~$A$ in the flipped tracer
      system.
      That is, in the tracer system, if the $X$-tracer is in state~$A$ and the $Y$-tracer is in state~$B$, 
      then both tracers switch states; 
      if the $X$-tracer is in state~$B$ and the $Y$-tracer is in state~$A$, then no
      interaction occurs.\label{i:priority}
  \end{enumerate}
  The interactions repeat until there are no $A$- and $B$-particles together on the site except
  possibly the two tracers with the prioritized tracer in state~$A$ and the nonprioritized tracer in state~$B$.
  Both tracers start in state~$A$ when $\xi_0(x)\geq 0$. If $\xi_0(x)=-1$, then the $X$-tracer
  starts in state~$B$ and the $Y$-tracer starts in state~$A$, and the $B$-particle at $x$
  is annihilated at time~$0$. If $\xi_0(x)\leq -2$, then both tracers start in state~$B$,
  and the two $B$-particles at $x$ with the greatest braveness are annihilated at time~$0$.

  In the tracer system, let $A_t^X$ and $A_t^Y$ be the events that the $X$- and $Y$-tracers, respectively, 
  are in state~$A$ at time $t$. Let $B_t^X$ and $B_t^Y$ be the complements of $A_t^X$
  and $A_t^Y$. We define the \emph{life} of each tracer up to time $T$ as 
\begin{align}
    L^X_T = \int_0^T \ind{A_t^X}\, dt  \qquad\qquad\text{and}\qquad\qquad L^Y_T &= \int_0^T \ind{A_t^Y}\, dt.\label{eq:life}
\end{align} 
These quantities represent the duration of the $X$ and $Y$ paths that each 
tracer has traveled along up to time $T$. 
Observe that the locations of the tracers at time~$t$ are $X_{L^X_t}$ and $Y_{L^Y_t}$.
Let $\widehat{A}_t^X$, $\widehat{A}_t^Y$, $\widehat{B}_t^X$, and $\widehat{B}_t^Y$ be the analogous
events for the flipped tracer system, and let the life of each tracer up to time~$t$ in the flipped system
be defined analogously and denoted by $\widehat{L}^X_t$ and $\widehat{L}^Y_t$.

In the tracer system let $\alpha_t(z)$ and $\beta_t(z)$ denote the number of nontracer $A$- and $B$-particles,
respectively, present on site~$z$ at time~$t>0$. We then define
\begin{align}
  \zeta_t(z) &= \alpha_t(z)  - \beta_t(z)
    - \1\bigl\{B_t^X\bigr\}\1\bigl\{X_{L_t^X}=z\bigr\} - \1\bigl\{B_t^Y\bigr\}\1\bigl\{Y_{L_t^Y}=z\bigr\},\label{eq:zeta}\\
  \zeta^X_t(z) &= \alpha_t(z)  - \beta_t(z) + \1\bigl\{A_t^X\bigr\}\1\bigl\{X_{L_t^X}=z\bigr\}
    - \1\bigl\{B_t^Y\bigr\}\1\bigl\{Y_{L_t^Y}=z\bigr\},\label{eq:zeta.X}\\
  \zeta^{X,Y}_t(z) &= \alpha_t(z)  - \beta_t(z) + \1\bigl\{A_t^X\bigr\}\1\bigl\{X_{L_t^X}=z\bigr\}
    + \1\bigl\{A_t^Y\bigr\}\1\bigl\{Y_{L_t^Y}=z\bigr\}.\label{eq:zeta.X.Y}
\end{align}
In short, $\zeta_t$ gives the particle counts for the tracer system with the tracer particles 
ignored in state~$A$ and counted as $B$-particles in state~$B$. Then $\zeta_t^X$ does the same
except it counts the $X$-tracer only in state~$A$, while $\zeta_t^{X,Y}$ counts both tracers only
while in state~$A$.

Finally, with $\widehat{\alpha}_t(z)$ and $\widehat{\beta}_t(z)$ denoting the counts of nontracer
$A$- and $B$-particles at site~$z$ in the flipped tracer system, we define
\begin{align}
  \widehat\zeta_t(z) &= \widehat\alpha_t(z)  - \widehat\beta_t(z)
     - \1\bigl\{\widehat B_t^Y\bigr\}\1\bigl\{Y_{\widehat L_t^Y}=z\bigr\}- \1\bigl\{\widehat B_t^X\bigr\}\1\bigl\{X_{\widehat L_t^X}=z\bigr\},\label{eq:zetarev}\\
  \widehat \zeta^Y_t(z) &= \widehat\alpha_t(z)  - \widehat\beta_t(z) + \1\bigl\{\widehat A_t^Y\bigr\}\1\bigl\{Y_{\widehat L_t^Y}=z\bigr\}
    - \1\bigl\{\widehat B_t^X\bigr\}\1\bigl\{X_{\widehat L_t^X}=z\bigr\},\label{eq:zetarev.Y}\\
  \widehat \zeta^{Y,X}_t(z) &= \widehat\alpha_t(z)  - \widehat\beta_t(z) 
    + \1\bigl\{\widehat A_t^Y\bigr\}\1\bigl\{Y_{\widehat L_t^Y}=z\bigr\}+\1\bigl\{\widehat A_t^X\bigr\}\1\bigl\{X_{\widehat L_t^X}=z\bigr\}.\label{eq:zetarev.Y.X}
\end{align}
Note the reversal of the roles of $X$ and $Y$ in \eqref{eq:zetarev}--\eqref{eq:zetarev.Y.X} as compared to \eqref{eq:zeta}--\eqref{eq:zeta.X.Y}.\noeqref{eq:zeta.X}\noeqref{eq:zetarev.Y}

\subsection{Identities}

We write $\xi^x$ for the system with the initial condition at $x$ increased by one $A$-particle so that $\alpha^x_0(x) = \alpha_0(x) +1$, $\alpha_0^x(z) = \alpha_0(z)$ for $z \neq x$, and $\beta_0^x = \beta_0$. Let $\xi^{x,x} = (\xi^x)^x$.
The relevance of the tracer system we have defined is that $\zeta$, $\zeta^X$, $\zeta^{X,Y}$,
$\widehat\zeta$, $\widehat\zeta^Y$, and $\widehat\zeta^{Y,X}$ all represent particle counts for DLAS,
as we will see in \thref{prop:zeta}. Both $\zeta$ and $\widehat\zeta$ turn out
to be identical to $\xi$. The counts $\zeta^X$ and $\widehat\zeta^Y$ are both instances of DLAS
with initial configuration $\xi_0^x$, but in $\zeta^X$ the extra particle compared to $\xi$ follows
path $X$, while in $\widehat\zeta^Y$ it follows path $Y$ (see \thref{lem:difference}).
Meanwhile $\zeta^{X,Y}$ and $\widehat\zeta^{Y,X}$ are both instances of DLAS with initial
configuration $\xi_0^{x,x}$, but they will differ slightly from each other because of prioritizing
paths $X$ and $Y$ differently.

\begin{proposition}\ \thlabel{prop:zeta}
  \begin{enumerate}[(I)]
    \item $\zeta_t(z)=\widehat\zeta_t(z)=\xi_t(z)$ for all $t>0$ and $z\in G$; \label{i:zeta}
    \item $(\zeta^X_t(z))_{t>0,z\in G} \eqd (\widehat\zeta^Y_t(z))_{t>0,z\in G}\eqd (\xi_t^x(z))_{t>0,z\in G}$; \label{i:zeta.X}
    \item $(\zeta^{X,Y}_t(z))_{t>0,z\in G} \eqd (\widehat\zeta^{Y,X}_t(z))_{t>0,z\in G}\eqd (\xi_t^{x,x}(z))_{t>0,z\in G}$. \label{i:zeta.X.Y}
  \end{enumerate}
\end{proposition}
\begin{proof}

  
  We start by proving $\zeta^X\eqd\xi^x$ and $\zeta^{X,Y}\eqd\xi^{x,x}$.
  First, observe that $\zeta^X$, and $\zeta^{X,Y}$ have initial configurations matching $\xi^x$
  and $\xi^{x,x}$.
  Next, we consider the effect of a jump in the tracer system from the perspective
  of $\zeta^X$ and $\zeta^{X,Y}$. All possibilities when a nontracer $A$-particle
  jumps are depicted in Figure~\ref{fig:nontracer.jump} (located in the Appendix). Note that if the random walk paths are in discrete time and multiple $A$-particles arrive simultaneously, what is depicted is the last collision to be resolved with the least brave arriving $A$-particle. Examining the figure, we see
  that when such a jump occurs, the counts given
  by $\zeta_t^X$ and $\zeta_t^{X,Y}$ evolve according to the rules of DLAS.
  The possibilities when a tracer particle in state~$A$ jumps are shown in Figure~\ref{fig:tracer.jump} (located in the Appendix).
  Again, in all cases the counts given
  by $\zeta_t^X$ and $\zeta_t^{X,Y}$ evolve according to the rules of DLAS;
  the only difference with a tracer particle jumping is that in some of the cases for $\zeta^X$, 
  a jump results in no change.
  Since $\zeta^X$, and $\zeta^{X,Y}$ have initial distributions matching
  $\xi^x$ and $\xi^{x,y}$ and evolve according to the same rules, their laws match as well.
  
  Next, we show that $\zeta$ and $\xi$ are equal a.s. We do so by showing that
  the particle system counted by $\xi$ at all times matches the nontracer particles 
  in the tracer system along with
  an additional $B$-particle at the location of any tracer particle in state~$B$.
  This is initially true by definition of the tracer process. We claim that after every jump,
  it continues to hold. Indeed, we just check that this is true for each interaction type 
  \ref{i:nontracers}--\ref{i:priority} from the definition of the tracer system.
  
  Finally, we observe that $\widehat\zeta$, $\widehat\zeta^Y$, and $\widehat\zeta^{Y,X}$
  are defined identically to $\zeta$, $\zeta^X$, and $\zeta^{X,Y}$ except that the roles
  of $X$ and $Y$ are reversed. Thus $\zeta\eqd\widehat\zeta$, $\zeta^X\eqd\widehat\zeta^Y$,
  and $\zeta^{X,Y}\eqd\widehat\zeta^{Y,X}$, since the paths $X$ and $Y$ are i.i.d.
  And the almost sure equality of $\widehat\zeta$ and $\xi$ holds by the same proof as for $\zeta$
  and $\xi$.
\end{proof}

\begin{remark}\thlabel{rmk:CRS.comparison}
  The tracer and flipped tracer systems are closely related to the 
\emph{dragged tracer} construction of \cite[Section~4.1]{CRS}.
In the terminology of \cite{CRS}, a tracer in state~$A$ is \emph{following an $A$-particle} and a tracer
in state~$B$ is \emph{following a $B$-particle}.
The difference between our construction and the one in \cite{CRS} is rule~\ref{i:priority}, the prioritization of 
one tracer over the other. In the dragged tracer construction, a tracer following an $A$-particle does not
interact with one following a $B$-particle. Using this construction, parts~\ref{i:zeta} and \ref{i:zeta.X.Y}
of \thref{prop:zeta} hold, but part~\ref{i:zeta.X} fails. Essentially, our construction
allows us to simultaneously couple the systems with zero, one, and two particles added at a site.
Without it, we can couple any two of these systems, but not all three together.
\end{remark}

The tracers are so-called because they track the differences between these systems:
\begin{lemma}\thlabel{lem:difference}
  For all $t>0$ and $z\in G$,
  \begin{align}
    \zeta^{X}_t(z) - \zeta_t(z) &= \ind{X_{L_t^X}=z}, & \zeta^{X,Y}_t(z) - \zeta_t(z) &= \ind{X_{L_t^X}=z} + \ind{Y_{L_t^Y}=z},\\
    \widehat \zeta^{Y}_t(z) - \widehat\zeta_t(z) &= \ind{Y_{\widehat L_t^Y}=z}, 
     &
    \widehat\zeta^{Y,X}_t(z) - \widehat\zeta_t(z) &= \ind{Y_{\widehat L_t^Y}=z} + \ind{X_{\widehat L_t^X}=z}.
  \end{align}
\end{lemma}
\begin{proof}
  These facts are direct consequences of definitions \eqref{eq:zeta}--\eqref{eq:zeta.X.Y}
  and \eqref{eq:zetarev}--\eqref{eq:zetarev.Y.X}.
\end{proof}

\thref{lem:difference} highlights that the tracer system tracks what happens
when we add the $X$-tracer first and the $Y$-tracer second, 
while the flipped tracer system tracks what happens when the tracers are added in the opposite
order. To prove that the monotonicity condition \ref{i:Phi.2nd.order} from \thref{def:icx.stat}
holds, we use the tracer system to track the effect of adding the $Y$-tracer after
the $X$-tracer is added, and we use the flipped tracer system to track the effect of
adding the $Y$-tracer alone. The following lemma is the basis of comparison between these two effects:
\begin{lemma}\thlabel{lem:longer.life}
  It holds for all $t>0$ that $L_t^Y\geq \widehat{L}_t^Y$.
\end{lemma}
\begin{proof}
  Let $T=\inf\{t\colon L_t^Y<\widehat{L}_t^Y\}$ and suppose by way of contradiction
  it is finite. Since $L_t^Y$ and $\widehat{L}_t^Y$ are continuous in $t$, we have $L_T^Y=\widehat{L}_T^Y$.
  Let $y=Y_{L_T^Y}=Y_{\widehat{L}_T^Y}$, the location of the $Y$-tracer at time~$T$
  in both the tracer and flipped tracer systems. By the definition of $T$, we have $L_T^Y\geq\widehat{L}_T^Y$ and that
  the $Y$-tracer must be in state~$B$ in the tracer system and state~$A$ in the flipped tracer system.
  Now, we argue this is a contradiction. By the dynamics of the tracer system, 
  the $Y$-tracer in state~$B$ may not
  sit on the same site as a nontracer $A$-particle.
  From the definition of $\zeta$ given in \eqref{eq:zeta},
  we have $\zeta_T(y)<0$. On the other hand, in the flipped tracer system, the $Y$-tracer in state~$A$
  may not sit on the same site as a nontracer $B$-particle or the $X$-tracer in state~$B$.
  From the definition of $\widehat{\zeta}$ given in \eqref{eq:zetarev}, we have $\widehat{\zeta}_T(y)\geq 0$.
  But by \thref{prop:zeta}~\ref{i:zeta} and its analogue for the flipped
  tracer system, we have $\zeta_T(y)=\xi_T(y)=\widehat{\zeta}_T(y)$,
  a contradiction.
\end{proof}

Recalling the definition of the occupation time $V_T$ from \eqref{eq:V}, it follows from \thref{prop:zeta}~\ref{i:zeta} that
\begin{align*}
  V_T &= \sum_{x \in H}\int_0^T \zeta_t(\rho)\1\bigl\{\zeta_t(\rho)>0\bigr\}dt,
\end{align*}
using the equality of $\xi_t$ and $\zeta_t$. Define
\begin{align}
  V_T^X &= \sum_{x \in H}\int_0^T \zeta^X_t(\rho)\1\bigl\{\zeta^X_t(\rho)>0\bigr\}dt,\\
  V_T^Y &= \sum_{x \in H}\int_0^T \widehat{\zeta}^Y_t(\rho)\1\bigl\{\widehat{\zeta}^Y_t(\rho)>0\bigr\}dt,\\
  V_T^{X,Y} &= \sum_{x \in H}\int_0^T \zeta^{X,Y}_t(\rho)\1\bigl\{\zeta^{X,Y}_t(\rho)>0\bigr\}dt. \label{eq:V^}
\end{align}
The random variables $V_T^X$ and $V_T^Y$ give the occupation time at sites~$H$ when
the $X$-tracer and $Y$-tracer, respectively, are added. Thus $V_T^{X}$ is defined in terms
of the tracer system while $V_T^{Y}$ is defined in terms of the flipped tracer system.
The occupation time $V_T^{X,Y}$ could be defined in terms of either system---if $\widehat\zeta^{Y,X}_t(\rho)$
replaced $\zeta^{X,Y}_t(\rho)$ in its definition, it would not change its distribution---but
its current definition is consistent with the proof strategy given in the paragraph
preceding \thref{lem:longer.life}.

\begin{lemma}\thlabel{lem:VT}
 For all $T>0$,
 \begin{align}
   V_T^X-V_T &= \int_0^{L_T^X} \ind{X_t\in H}\,dt,\label{eq:VTX.dif}\\
   V_T^Y - V_T &= \int_0^{\widehat{L}_T^Y} \ind{Y_t\in H}\,dt,\label{eq:VTY.dif}\\
   V_T^{X,Y} - V_T &= \int_0^{L_T^X} \ind{X_t \in H }\,dt + \int_0^{L_T^Y} \ind{Y_t\in H}\,dt.\label{eq:VTXY.dif}
 \end{align}
\end{lemma}
\begin{proof}
  We claim that for $x \in H$
  \begin{align}
    \zeta^X_t(\rho)\1\bigl\{\zeta^X_t(\rho)>0\bigr\} - \zeta_t(\rho)\1\bigl\{\zeta_t(\rho)>0\bigr\}
      &=\ind{X_{L^X_t}=\rho}\ind{\zeta_t(\rho)\geq 0}\label{eq:zx}\\
      &=\ind{X_{L^X_t}=\rho}\1_{A_t^X}.\label{eq:zx2}
  \end{align}
  Indeed, by \thref{lem:difference} we have $\zeta^X_t(\rho)-\zeta_t(\rho)=\ind{X_{L^X_t}=\rho}$.
  If $\zeta_t(\rho)<0$, then $\zeta_t^X(\rho)\leq 0$, and both terms on the left-hand side
  of \eqref{eq:zx} are zero.
  If $\zeta_t(\rho)\geq 0$, then $\zeta_t^X(\rho)\geq 0$, and \eqref{eq:zx} is equal to $\ind{X_{L^X_t}=\rho}$.
  To show \eqref{eq:zx2}, assume the $X$-tracer is at $\rho$.
  If it is in state~$B$, then there can be no nontracer $A$-particles at $\rho$, and from \eqref{eq:zeta}
  we have $\zeta_t(z)<0$. If it is in state~$A$, then there can be no nontracer $B$-particles
  at $\rho$, and by \ref{i:priority} the $Y$-tracer cannot
  be at $\rho$ in state~$B$; hence $\zeta_t(\rho)\geq 0$ from \eqref{eq:zeta}. Thus $A_t^X$ occurs
  if and only if $\zeta_t(\rho)\geq 0$ under the assumption that $X_{L_t^X}=\rho$. 
  
  Applying \eqref{eq:zx} and \eqref{eq:zx2},
  \begin{align*}
    V_T^X - V_T =\sum_{x \in H} \int_0^T \ind{X_{L^X_t}=\rho}\1_{A_t^X}\,dt.
  \end{align*}
  Since $L_t^X$ only increases while the $X$-tracer is in state~$A$, we arrive at \eqref{eq:VTX.dif}.
  Equation~\eqref{eq:VTY.dif} is proven identically.
  
  The proof of \eqref{eq:VTXY.dif} is similar. We argue that
  \begin{align}
    \begin{split}
    &\zeta^{X,Y}_t(\rho)\1\bigl\{\zeta^{X,Y}_t(\rho)>0\bigr\} - \zeta_t(\rho)\1\bigl\{\zeta_t(\rho)>0\bigr\}\\
      &\qquad\qquad\qquad\qquad= \1\bigl\{X_{L_t^X}=\rho\bigr\}\1_{A_t^X}
        +\1\bigl\{Y_{L_t^Y}=\rho\bigr\}\1_{A_t^Y}.
    \end{split}\label{eq:zxy}
  \end{align}
  There are three cases to consider. First, if $\zeta_t(\rho)\geq 0$, then $\zeta_t^{X,Y}(\rho)\geq 0$
  by \thref{lem:difference}. Thus the left-hand side of \eqref{eq:zxy} is equal to
  $\ind{X_{L_t^X}=z} + \ind{Y_{L_t^Y}=z}$ in this case by \thref{lem:difference}.
  Next, if $\zeta_t(\rho)\leq -2$, then both terms on the left-hand side of \eqref{eq:zxy}
  are zero. 
  Last, if $\zeta_t(\rho)=-1$ then $\zeta_t^{X,Y}$ is one of $-1$, $0$, and $1$, and
  hence the left-hand side of \eqref{eq:zxy}
  to $\ind{\zeta^{X,Y}_t(\rho)=1}$. All together, we have shown that
  \begin{align*}
    &\zeta^{X,Y}_t(\rho)\1\bigl\{\zeta^{X,Y}_t(\rho)>0\bigr\} - \zeta_t(\rho)\1\bigl\{\zeta_t(\rho)>0\bigr\}\\
      &\qquad=\bigl(\ind{X_{L_t^X}=\rho} + \ind{Y_{L_t^Y}=\rho}  \bigr)\ind{\zeta_t(\rho)\geq 0} +
      \ind{\zeta^{X,Y}_t(\rho)=1,\, \zeta_t(\rho)=-1}
  \end{align*}
  The event that $\zeta^{X,Y}_t(\rho)=1$ and $\zeta_t(\rho)=-1$ can only occur when 
  both tracers are at $\rho$ by \thref{lem:difference}. If both tracers are in state~$A$ then
  $\zeta_t(z)\geq 0$, while if both tracers are in state~$B$ then $\zeta_t^{X,Y}(\rho)\leq 0$.
  Hence this event occurs when both tracers are at $\rho$ with the $X$-tracer in state~$B$,
  the $Y$-tracer in state~$A$, and (necessarily) no nontracer particles are at $\rho$.
  That is,
  \begin{align*}
    \ind{\zeta^{X,Y}_t(\rho)=1,\, \zeta_t(\rho)=-1} = \ind{X_{L_t^X}=\rho,\,Y_{L_t^Y}=\rho}\1_{B_t^X}\1_{A_t^Y}.
  \end{align*}
  Meanwhile, when proving \eqref{eq:zx2}, we showed that
  $A_t^X$ occurs if and only if $\zeta_t(\rho)\geq 0$ under the assumption that $X_{L_t^X}=\rho$.
  All together, the left-hand side of \eqref{eq:zxy} is equal to
  \begin{align*}
    &\ind{X_{L_t^X}=\rho}\1_{A_t^X} + \ind{Y_{L_t^Y}=\rho}\ind{\zeta_t(\rho)\geq 0}
      + \ind{X_{L_t^X}=\rho,\,Y_{L_t^Y}=\rho}\1_{B_t^X}\1_{A_t^Y}\\
      &\qquad =\ind{X_{L_t^X}=\rho}\1_{A_t^X} + \ind{Y_{L_t^Y}=\rho}\bigl(\ind{\zeta_t(\rho)\geq 0} 
      + \ind{X_{L_t^X}=\rho}\1_{B_t^X}\1_{A_t^Y}\bigr)\\
      &\qquad = \ind{X_{L_t^X}=\rho}\1_{A_t^X} + \ind{Y_{L_t^Y}=\rho}\1_{A_t^Y}.
  \end{align*}
  With \eqref{eq:zxy} proven, the rest of the proof of \eqref{eq:VTXY.dif} goes the same as for
  \eqref{eq:VTX.dif} and \eqref{eq:VTY.dif}.
  \end{proof}

  \begin{remark} \thlabel{rem:B}
If $B$-particles moved, then similar equations to those in \eqref{eq:zeta}--\eqref{eq:zetarev.Y.X} could be derived. However the time change would be more involved since parts of $X$ and $Y$ would be traversed by the tracer in state $B$. This time change would also complicate the formulas in \thref{lem:VT}, since the integrals would be over disconnected subintervals within $[0,T]$ traversed by the tracers while in state $A$, rather than a single connected subinterval such as $[0,L_T^X]$ when $B$-particles do not move. 
Nonetheless, the approach used in \cite[Section 4]{CRS} to track the difference between systems with different starting configurations when the jump rates of $A$- and $B$-particles are possibly distinct ought to apply to our similarly constructed tracers. Hence \thref{prop:zeta} and by extension \thref{lem:difference} should both continue to hold.

Issues arise with \thref{lem:longer.life} and \thref{lem:VT}. The lives of, say, the $Y$-tracer in the tracer and flipped tracer systems are given by integrals over possibly different collections of subintervals. Coupling the size of these subintervals and the tracer locations at these times does not seem possible. Consequently, we do not believe that the conclusions of \thref{lem:longer.life} nor \thref{lem:VT} generalize to the setting in which $B$-particle are also mobile. 

\end{remark}

\subsection{Proof of \thref{prop:i_ii}}

\begin{proof}
  Suppose that $(\xi_0, S,h)$ is regular with $\sum_{x \in G} |\xi_0(x)|< \infty$. Fix $x$ and $k$ and define the tracer and flipped tracer systems with initial configuration
  $\xi_{0,k}$. Let $V_T$ be the occupation time of $H \subseteq G$ for $(\xi_{0,k},S,h)$. 
  We claim that the coupled random variables $(V_T, V_T^X, V_T^Y V_T^{X,Y})$ defined at \eqref{eq:V^}
  satisfy the conditions of \thref{def:icx.stat}.
  By \thref{prop:zeta}, conditions~\ref{i:Phi}--\ref{i:Phi.X.Y} hold.
  By \eqref{eq:VTX.dif} and \eqref{eq:VTY.dif} in \thref{lem:VT}, condition~\ref{i:Phi.1st.order}
  holds. By \thref{lem:VT} again,
  \begin{align*}
    V_T^{X,Y} - V_T^X - V_T^Y + V_T &= (V_T^{X,Y}-V_T) - (V_T^X-V_T) - (V_T^Y-V_T)\\
     &= \int_0^{L_T^Y}\ind{Y_t\in H}\,dt - \int_0^{\widehat L_T^Y} \ind{Y_t\in H}\,dt,
  \end{align*}
  and this is nonnegative since $L_T^Y\geq \widehat L_T^Y$ by \thref{lem:longer.life},
  which shows that condition~\ref{i:Phi.2nd.order} holds.
\end{proof}

\section{Monotonicity of icx statistics under the icx order} 

This section connects icx statistics to the icx order. 
The idea is to increase the initial configuration in the icx order at one site at a time,
and to show that at each step the icx statistic increases as well. It is more a technical argument
about stochastic orders than it is anything about diffusion-limited annihilating systems, and it uses
similar arguments from \cite[Lemma~17]{JJ} as templates.

\begin{proposition} \thlabel{prop:icx} 
Suppose that \eqref{eq:comparable} holds. Further assume that $\sum_{z\in G} |\xi_0(z)| < \infty$ and $\sum_{z \in G}| \xi_0'(z)| < \infty$ almost surely.  Fix $x \in G$ and assume that $\xi_0(z) = \xi_0'(z)$ for all $z \neq x$.  Suppose that $f$ is a nonnegative icx statistic of $(\xi_0, S,h)$ as in \thref{def:icx.stat} and let $f' = f(\xi_0', S, h)$.
    \begin{center}If $\xi_0(x) \pmicx \xi_0'(x)$, then  $f \preceq_{icx} f'$.
    \end{center}
\end{proposition}

For a function~$h$ on the integers, we define the difference operators
\begin{align}\label{eq:D.def}
  D h(k) &= h(k+1)-h(k)\\
  D^2 h(k) &= D[h(k+1)-h(k)] = h(k+2) - 2h(k+1) + h(k).
\end{align}
The discrete analogue of convexity is that $f$ satisfies $Df(k)\geq 0$ and $D^2f(k)\geq 0$ for 
all integers $k$. We would expect that if $X\icx Y$ for integer-valued random variables
$X$ and $Y$, then $\E h(X)\leq \E h(Y)$ when $h$ is convex. Indeed, this is correct:

\begin{lemma} \thlabel{lem:test}
 Let $X$ and $Y$ be integer-valued random variables and $h \colon \mathbb Z \to \mathbb R$ satisfy $D h(k) \geq 0$ and $D^2 h(k) \geq 0$ for all $k \in \mathbb Z$. If $X\pmicx Y$, then $\E h(X) \leq \E h(Y)$.

\end{lemma}
\begin{proof}
  Let $\bar h \colon (-\infty, \infty) \to \mathbb R$ be the linear interpolation of $h$ between adjacent integer points. Since $Dh(k) \geq 0$ and $D^2 h(k) \geq 0$, the function $\bar h$ is increasing and convex on $(-\infty, \infty)$. Thus, $X \icx Y$ implies that $\E h(X) =\E \bar h(X) \leq \E \bar h(Y) = \E h(Y)$.
\end{proof}


\begin{proof}[Proof of \thref{prop:icx}]
  Define $\xi_{0,k}$ to be the same as $\xi_0$ except that $\xi_{0,k} (x) = k$. Let $W(k) = f(\xi_{0,k},S,h)$ so that
\begin{align}
    W(\xi_0(x)) \overset{d}= f \quad \text{and} \quad W(\xi_0'(x)) \overset{d}= f'.
\end{align}
 
Let $\varphi \colon [0, \infty) \rightarrow [0, \infty)$ be an increasing convex function and let $h(k) = \E \varphi(W(k))$ for all $k \in \mathbb Z$. Then
\begin{align}\label{eq:testfunlink}
  \E h(\xi_0(x)) = \E\varphi(f) \quad \text{and} \quad \E h(\xi_0'(x)) = \E \varphi(f').
\end{align}
We will show that $Dh(k)\geq 0$ and $D^2h(k)\geq 0$.
Then it will follow from \thref{lem:test} that 
$\E h(\xi_0(x))\leq\E h(\xi_0'(x))$, since $\xi_0(x)\pmicx \xi_0'(x)$ by hypothesis. 
By \eqref{eq:testfunlink}, this shows that
$\E\varphi(f)\leq\E\varphi(f')$, which proves $f\icx f'$.

Thus it only remains to show $Dh(k)\geq 0$ and $D^2h(k)\geq 0$.
Let $(\Phi,\Phi^X,\Phi^Y,\Phi^{X,Y})$ be a coupling as described in the definition of an icx statistic.
Then
  \begin{align*}
    D h(k) = \E \varphi(f(\xi_{0,k+1},S,h)) - \E \varphi(f(\xi_{0,k},S,h)) = \E [ \varphi(\Phi^X)-\varphi(\Phi)].
  \end{align*}
  Since $\Phi^X\geq\Phi$ a.s.\ and $\varphi$ is increasing, we have $D h(k)\geq 0$. 

  For the second order condition, we expand $D^2h(k)$ as
  \begin{align}
    D^2 h(k) &= \E\bigl[ \varphi(f(\xi_{0,k+2},S,h))\bigr] - 2\E\bigl[\varphi(f(\xi_{0,k+1},S,h))\bigr] + 
      \E\bigl[\varphi(f(\xi_{0,k},S,h))\bigr] \\
      &= \E\bigl[ \varphi(\Phi^{X,Y}) - \varphi(\Phi^X) - \varphi(\Phi^Y) + \varphi(\Phi) \bigr].
                           \label{eq:diff}
  \end{align}    
  We claim that
  \begin{align*}
    \varphi(\Phi^{X,Y}) - \varphi(\Phi^X)&\geq \varphi(\Phi^X + \Phi^Y-\Phi) - \varphi(\Phi^X)\\
      &\geq \varphi(\Phi^Y)-\varphi(\Phi).
  \end{align*}
  The first inequality holds because $\varphi$ is increasing and
   $\Phi^{X,Y}\geq \Phi^X+\Phi^Y-\Phi$ by item~\ref{i:Phi.2nd.order} of the definition of icx statistic.
  For the  second inequality, observe that by convexity
  $\varphi(a+u)-\varphi(a)\geq \varphi(b+u)-\varphi(b)$ for any $a\geq b$
  and $u\geq 0$, and then take $a=\Phi^X$, $b=\Phi$, and $u=\Phi^Y-\Phi$.
  Thus \eqref{eq:diff} is nonnegative, completing the proof.
\end{proof}

\section{Further questions} \label{sec:further}

A natural question is whether or not our results extend to DLAS with mobile $B$-particles.
Our intuition is that they do, since it remains the case that two $A$-particles can assist 
each other to clear out $B$-particles that would otherwise destroy an $A$-particle added individually.
But as we discuss in \thref{rem:B}, if we use our current coupling with moving $B$-particles,
the tracer and flipped tracer systems have a more complex relationship, and we cannot prove
that $V_T$ is an icx statistic.

Even for the stationary $B$-particle case, it would be interesting to find functionals besides occupation time that respect the icx order. For example, we speculate that the lifespan of a distinguished $A$- or $B$-particle might be monotonic (increasing or decreasing, respectively) in the icx order as more $A$-particles are added. Similarly, the time of the first visit to a distinguished vertex by an $A$-particle as well as the total number of $A$-particles still alive at time $t$ might respect the icx order. 

Finally, we are interested in whether there is a general framework for understanding which statistics
in different interacting particle systems respect which stochastic orders.
This paper demonstrates that DLAS, with interaction rule $A+B\to\varnothing$, is compatible
with the icx order. In \cite{JJ}, it is shown that the frog model, with interaction rule $A+B\to 2A$,
is compatible with the icv order. Is there a systematic explanation across different particle systems?



\newpage 

\appendix 

\section{Figures from the proof of \thref{prop:zeta}}
\quad 

\begin{figure}[h]
    \centering
\begin{tikzpicture}[particle/.style={anchor=base west,shift={(-.175,.1)},inner sep=0}]
      \begin{scope}[shift={(3,0)}]
        \path (-3.6,0) node[text width=6.25cm,anchor=base west] 
           {Case 1: The $A$-particle jumps onto a site containing neither $B$-particles
            nor tracers in state~$B$. No interaction occurs.};
        \begin{scope}[shift={(0,-.2)}]
        \draw[thick,shift={(-1.75,0)}] (-.1,-2)--+(-.6,0) (-.4,-2) node[particle] (v1) {$A$} (.1,-2)--+(.6,0) (.4,-2) node[particle] (v2) {$A^X$}             ;
        \draw[thick,->] (v1) to [out=45,in=135] (v2);
        \draw[thick,shift={(1.75,0)}] (-.1,-2)--+(-.6,0)  (.1,-2)--+(.6,0) (.4,-2) node[particle] (av2) {$A^X$}         node[particle,shift={(0,.35)}] (av1) {$A$}    ;
        \draw[decoration=snake,decorate,->] (-.75,-1.9)--(.75,-1.9);
        
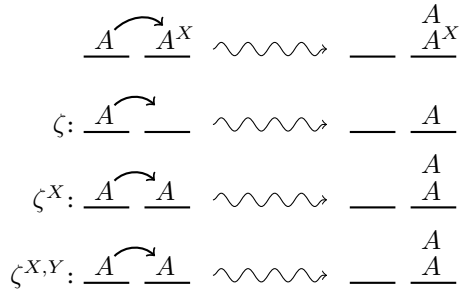
        \begin{scope}[shift={(0.0,-1)}]
          \draw[thick,shift={(-1.75,0)}] (-.1,-2)--++(-.6,0) node[anchor=base east] {$\zeta$:}(-.4,-2) node[particle] (v1) {$A$} (.1,-2)--+(.6,0) (.4,-2) node[particle] (v2) {$\phantom{A}$}             ;
        \draw[thick,->] (v1) to [out=45,in=135] (v2);
        \draw[thick,shift={(1.75,0)}] (-.1,-2)--+(-.6,0)  (.1,-2)--+(.6,0) (.4,-2) node[particle] (av2) {$A$};
        \draw[decoration=snake,decorate,->] (-.75,-1.9)--(.75,-1.9);

        \begin{scope}[shift={(0.0,-1)}]
          \draw[thick,shift={(-1.75,0)}] (-.1,-2)--++(-.6,0) node[anchor=base east] {$\zeta^X$:}(-.4,-2) node[particle] (v1) {$A$} (.1,-2)--+(.6,0) (.4,-2) node[particle] (v2) {$A$}             ;
        \draw[thick,->] (v1) to [out=45,in=135] (v2);
        \draw[thick,shift={(1.75,0)}] (-.1,-2)--+(-.6,0)  (.1,-2)--+(.6,0) (.4,-2) node[particle] (av2) {$A$}         node[particle,shift={(0,.35)}] (av1) {$A$}    ;
        \draw[decoration=snake,decorate,->] (-.75,-1.9)--(.75,-1.9);

        \begin{scope}[shift={(0,-1)}]
          \draw[thick,shift={(-1.75,0)}] (-.1,-2)--++(-.6,0) node[anchor=base east] {$\zeta^{X,Y}$:}(-.4,-2) node[particle] (v1) {$A$} (.1,-2)--+(.6,0) (.4,-2) node[particle] (v2) {$A$}             ;
        \draw[thick,->] (v1) to [out=45,in=135] (v2);
        \draw[thick,shift={(1.75,0)}] (-.1,-2)--+(-.6,0)  (.1,-2)--+(.6,0) (.4,-2) node[particle] (av2) {$A$}         node[particle,shift={(0,.35)}] (av1) {$A$}    ;
        \draw[decoration=snake,decorate,->] (-.75,-1.9)--(.75,-1.9);
        \end{scope}
        \end{scope}
        \end{scope}
        \end{scope}

      \end{scope}
      \begin{scope}[shift={(9.5,0)}]
        \path (-3.6,0) node[text width=6.25cm,anchor=base west] {Case 2: The $A$-particle jumps onto a site containing nontracer $B$-particles, mutually annihilating with one of them.};
              
        \begin{scope}[shift={(0,-.2)}]
          \draw[thick,shift={(-1.75,0)}] (-.1,-2)--+(-.6,0) (-.4,-2) node[particle] (v1) {$A$} (.1,-2)--+(.6,0) (.4,-2) node[particle] (v2) {$B$} node[particle,shift={(0,.35)}] (av1) {$B^X$}            ;
        \draw[thick,->] (v1) to [out=45,in=135] (v2);
        \draw[thick,shift={(1.75,0)}] (-.1,-2)--+(-.6,0)  (.1,-2)--+(.6,0) (.4,-2) node[particle] (av2) {$B^X$}             ;
        \draw[decoration=snake,decorate,->] (-.75,-1.9)--(.75,-1.9);
        
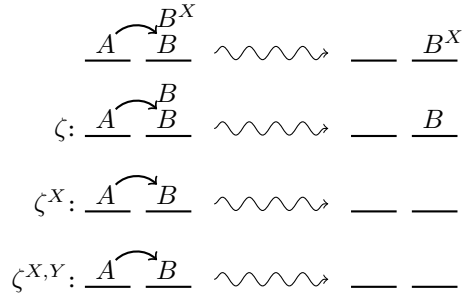
        \begin{scope}[shift={(0.0,-1)}]
          \draw[thick,shift={(-1.75,0)}] (-.1,-2)--++(-.6,0) node[anchor=base east] {$\zeta$:}(-.4,-2) node[particle] (v1) {$A$} (.1,-2)--+(.6,0) (.4,-2) node[particle] (v2) {$B$}   node[particle,shift={(0,.35)}] (av1) {$B$}          ;
        \draw[thick,->] (v1) to [out=45,in=135] (v2);
        \draw[thick,shift={(1.75,0)}] (-.1,-2)--+(-.6,0)  (.1,-2)--+(.6,0) (.4,-2) node[particle] (av2) {$B$};
        \draw[decoration=snake,decorate,->] (-.75,-1.9)--(.75,-1.9);

        \begin{scope}[shift={(0.0,-1)}]
          \draw[thick,shift={(-1.75,0)}] (-.1,-2)--++(-.6,0) node[anchor=base east] {$\zeta^X$:}(-.4,-2) node[particle] (v1) {$A$} (.1,-2)--+(.6,0) (.4,-2) node[particle] (v2) {$B$}             ;
        \draw[thick,->] (v1) to [out=45,in=135] (v2);
        \draw[thick,shift={(1.75,0)}] (-.1,-2)--+(-.6,0)  (.1,-2)--+(.6,0) (.4,-2) node[particle] (av2) {$\phantom{A}$}             ;
        \draw[decoration=snake,decorate,->] (-.75,-1.9)--(.75,-1.9);

        \begin{scope}[shift={(0,-1)}]
          \draw[thick,shift={(-1.75,0)}] (-.1,-2)--++(-.6,0) node[anchor=base east] {$\zeta^{X,Y}$:}(-.4,-2) node[particle] (v1) {$A$} (.1,-2)--+(.6,0) (.4,-2) node[particle] (v2) {$B$}             ;
        \draw[thick,->] (v1) to [out=45,in=135] (v2);
        \draw[thick,shift={(1.75,0)}] (-.1,-2)--+(-.6,0)  (.1,-2)--+(.6,0) (.4,-2) node[particle] (av2) {$\phantom{A}$}             ;
        \draw[decoration=snake,decorate,->] (-.75,-1.9)--(.75,-1.9);
        \end{scope}
        \end{scope}
        \end{scope}
        \end{scope}

      \end{scope}

      \begin{scope}[shift={(3,-6.5)}]
        \path (-3.6,0) node[text width=6.25cm,anchor=base west] {Case 3: The $A$-particle jumps onto a site containing no nontracer particles, the $Y$-tracer in state~$B$, and possibly the $X$-tracer in either state. The $A$-particle is annihilated and the $Y$-tracer enters state~$A$.};
              
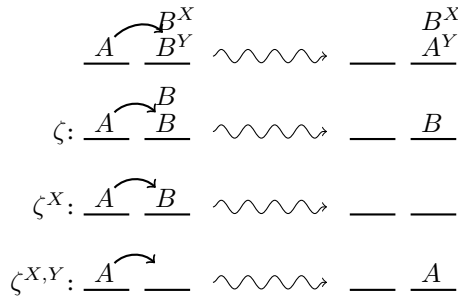
        \begin{scope}[shift={(0,-1.1)}]
          \draw[thick,shift={(-1.75,0)}] (-.1,-2)--+(-.6,0) (-.4,-2) node[particle] (v1) {$A$} (.1,-2)--+(.6,0) (.4,-2) node[particle] (v2) {$B^Y$} node[particle,shift={(0,.35)}] (av1) {$B^X$}            ;
        \draw[thick,->] (v1) to [out=45,in=135] (v2);
        \draw[thick,shift={(1.75,0)}] (-.1,-2)--+(-.6,0)  (.1,-2)--+(.6,0) (.4,-2) node[particle] (av2) {$A^Y$}         node[particle,shift={(0,.35)}] (av1) {$B^X$}            ;
        \draw[decoration=snake,decorate,->] (-.75,-1.9)--(.75,-1.9);
        
        \begin{scope}[shift={(0.0,-1)}]
          \draw[thick,shift={(-1.75,0)}] (-.1,-2)--++(-.6,0) node[anchor=base east] {$\zeta$:}(-.4,-2) node[particle] (v1) {$A$} (.1,-2)--+(.6,0) (.4,-2) node[particle] (v2) {$B$}  node[particle,shift={(0,.35)}] (av1) {$B$}   ;
        \draw[thick,->] (v1) to [out=45,in=135] (v2);
        \draw[thick,shift={(1.75,0)}] (-.1,-2)--+(-.6,0)  (.1,-2)--+(.6,0) (.4,-2) node[particle] (av2) {$B$};
        \draw[decoration=snake,decorate,->] (-.75,-1.9)--(.75,-1.9);

        \begin{scope}[shift={(0.0,-1)}]
          \draw[thick,shift={(-1.75,0)}] (-.1,-2)--++(-.6,0) node[anchor=base east] {$\zeta^X$:}(-.4,-2) node[particle] (v1) {$A$} (.1,-2)--+(.6,0) (.4,-2) node[particle] (v2) {$B$}             ;
        \draw[thick,->] (v1) to [out=45,in=135] (v2);
        \draw[thick,shift={(1.75,0)}] (-.1,-2)--+(-.6,0)  (.1,-2)--+(.6,0)             ;
        \draw[decoration=snake,decorate,->] (-.75,-1.9)--(.75,-1.9);

        \begin{scope}[shift={(0,-1)}]
          \draw[thick,shift={(-1.75,0)}] (-.1,-2)--++(-.6,0) node[anchor=base east] {$\zeta^{X,Y}$:}(-.4,-2) node[particle] (v1) {$A$} (.1,-2)--+(.6,0) (.4,-2) node[particle] (v2) {$\phantom{A}$}             ;
        \draw[thick,->] (v1) to [out=45,in=135] (v2);
        \draw[thick,shift={(1.75,0)}] (-.1,-2)--+(-.6,0)  (.1,-2)--+(.6,0) (.4,-2) node[particle] (v2) {$A$} ;
        \draw[decoration=snake,decorate,->] (-.75,-1.9)--(.75,-1.9);
        \end{scope}
        \end{scope}
        \end{scope}
        \end{scope}

      \end{scope}      

      \begin{scope}[shift={(9.5,-6.5)}]
        \path (-3.6,0) node[text width=6.25cm,anchor=base west] {Case 4: The $A$-particle jumps onto a site containing only the $X$-tracer in state~$B$. The $A$-particle is annihilated and the $X$-tracer enters state~$A$.};
              
        \begin{scope}[shift={(0,-1.1)}]
          \draw[thick,shift={(-1.75,0)}] (-.1,-2)--+(-.6,0) (-.4,-2) node[particle] (v1) {$A$} (.1,-2)--+(.6,0) (.4,-2) node[particle] (v2) {$B^X$};
        \draw[thick,->] (v1) to [out=45,in=135] (v2);
        \draw[thick,shift={(1.75,0)}] (-.1,-2)--+(-.6,0)  (.1,-2)--+(.6,0) (.4,-2) node[particle] (av2) {$A^X$}             ;
        \draw[decoration=snake,decorate,->] (-.75,-1.9)--(.75,-1.9);
        
        \begin{scope}[shift={(0.0,-1)}]
          \draw[thick,shift={(-1.75,0)}] (-.1,-2)--++(-.6,0) node[anchor=base east] {$\zeta$:}(-.4,-2) node[particle] (v1) {$A$} (.1,-2)--+(.6,0) (.4,-2) node[particle] (v2) {$B$}      ;
        \draw[thick,->] (v1) to [out=45,in=135] (v2);
        \draw[thick,shift={(1.75,0)}] (-.1,-2)--+(-.6,0)  (.1,-2)--+(.6,0) ;
        \draw[decoration=snake,decorate,->] (-.75,-1.9)--(.75,-1.9);

        \begin{scope}[shift={(0.0,-1)}]
          \draw[thick,shift={(-1.75,0)}] (-.1,-2)--++(-.6,0) node[anchor=base east] {$\zeta^X$:}(-.4,-2) node[particle] (v1) {$A$} (.1,-2)--+(.6,0) (.4,-2) node[particle] (v2) {$\phantom{B}$}             ;
        \draw[thick,->] (v1) to [out=45,in=135] (v2);
        \draw[thick,shift={(1.75,0)}] (-.1,-2)--+(-.6,0)  (.1,-2)--+(.6,0) (.4,-2) node[particle] (av2) {$A$}             ;
        \draw[decoration=snake,decorate,->] (-.75,-1.9)--(.75,-1.9);

        \begin{scope}[shift={(0,-1)}]
          \draw[thick,shift={(-1.75,0)}] (-.1,-2)--++(-.6,0) node[anchor=base east] {$\zeta^{X,Y}$:}(-.4,-2) node[particle] (v1) {$A$} (.1,-2)--+(.6,0) (.4,-2) node[particle] (v2) {$\phantom{B}$}             ;
        \draw[thick,->] (v1) to [out=45,in=135] (v2);
        \draw[thick,shift={(1.75,0)}] (-.1,-2)--+(-.6,0)  (.1,-2)--+(.6,0) (.4,-2) node[particle] (av2) {$A$}             ;
        \draw[decoration=snake,decorate,->] (-.75,-1.9)--(.75,-1.9);
        \end{scope}
        \end{scope}
        \end{scope}
        \end{scope}

      \end{scope}

\end{tikzpicture}
    
    \caption{When a nontracer $A$-particle in the tracer system jumps, there are four cases.
    The top line in each case shows how the tracer system evolves when the particle jumps.
    The tracer particles are indicated by $A^X$, $A^Y$, $B^X$, or $B^Y$, with $A$ and $B$ giving
    their states and $X$ and $Y$ specifying the tracer. The lines below show how the particles
    are viewed by $\zeta$, $\zeta^X$, and $\zeta^{X,Y}$. 
    }\label{fig:nontracer.jump}
  \end{figure}
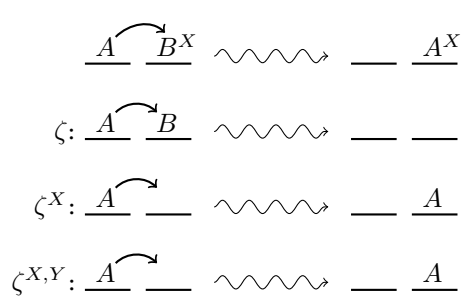
  
  \begin{figure}[h]
    \centering
\begin{tikzpicture}[particle/.style={anchor=base west,shift={(-.175,.1)},inner sep=0}]
      \begin{scope}[shift={(3,0)}]
        \path (-3.6,0) node[text width=6.25cm,anchor=base west] 
           {Case 1: The $X$-tracer in state~$A$ jumps onto a site containing
           no nontracer $B$-particles nor the $Y$-tracer in state~$B$. No interaction occurs.};
        \begin{scope}[shift={(0,-.2)}]
        \draw[thick,shift={(-1.75,0)}] (-.1,-2)--+(-.6,0) (-.4,-2) node[particle] (v1) {$A^X$} (.1,-2)--+(.6,0) (.4,-2) node[particle] (v2) {$A^Y$}             ;
        \draw[thick,->] (v1) to [out=45,in=135] (v2);
        \draw[thick,shift={(1.75,0)}] (-.1,-2)--+(-.6,0)  (.1,-2)--+(.6,0) (.4,-2) node[particle] (av2) {$A^Y$}         node[particle,shift={(0,.35)}] (av1) {$A^X$}    ;
        \draw[decoration=snake,decorate,->] (-.75,-1.9)--(.75,-1.9);
        
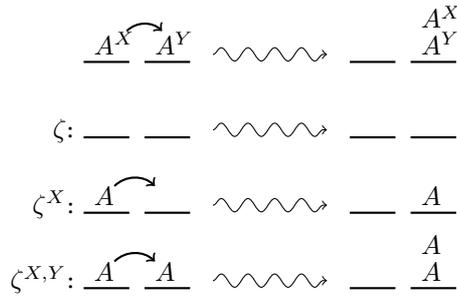
        \begin{scope}[shift={(0.0,-1)}]
          \draw[thick,shift={(-1.75,0)}] (-.1,-2)--++(-.6,0) node[anchor=base east] {$\zeta$:}(-.4,-2) node[particle] (v1) {$\phantom{A}$} (.1,-2)--+(.6,0) (.4,-2) node[particle] (v2) {$\phantom{A}$}             ;
        \draw[thick,shift={(1.75,0)}] (-.1,-2)--+(-.6,0)  (.1,-2)--+(.6,0) ;
        \draw[decoration=snake,decorate,->] (-.75,-1.9)--(.75,-1.9);

        \begin{scope}[shift={(0.0,-1)}]
          \draw[thick,shift={(-1.75,0)}] (-.1,-2)--++(-.6,0) node[anchor=base east] {$\zeta^X$:}(-.4,-2) node[particle] (v1) {$A$} (.1,-2)--+(.6,0) (.4,-2) node[particle] (v2) {$\phantom{A}$}             ;
        \draw[thick,->] (v1) to [out=45,in=135] (v2);
        \draw[thick,shift={(1.75,0)}] (-.1,-2)--+(-.6,0)  (.1,-2)--+(.6,0)  (.4,-2) node[particle] (v2) {$A$}       ;
        \draw[decoration=snake,decorate,->] (-.75,-1.9)--(.75,-1.9);

        \begin{scope}[shift={(0,-1)}]
          \draw[thick,shift={(-1.75,0)}] (-.1,-2)--++(-.6,0) node[anchor=base east] {$\zeta^{X,Y}$:}(-.4,-2) node[particle] (v1) {$A$} (.1,-2)--+(.6,0) (.4,-2) node[particle] (v2) {$A$} ;
        \draw[thick,->] (v1) to [out=45,in=135] (v2);
        \draw[thick,shift={(1.75,0)}] (-.1,-2)--+(-.6,0)  (.1,-2)--+(.6,0) (.4,-2) node[particle] (av2) {$A$}       node[particle,shift={(0,.35)}] (av1) {$A$}       ;
        \draw[decoration=snake,decorate,->] (-.75,-1.9)--(.75,-1.9);
        \end{scope}
        \end{scope}
        \end{scope}
        \end{scope}

      \end{scope}
      \begin{scope}[shift={(9.5,0)}]
        \path (-3.6,0) node[text width=6.25cm,anchor=base west] {Case 2: The $X$-tracer in state~$A$ jumps onto a site containing nontracer $B$-particles. It annihilates one of them and then enters state~$B$.};
              
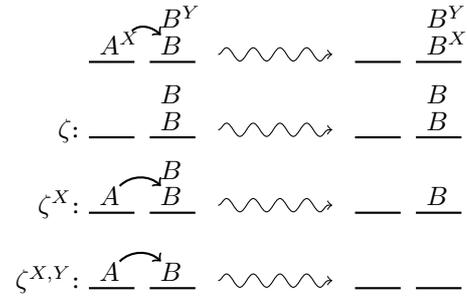
        \begin{scope}[shift={(0,-.2)}]
          \draw[thick,shift={(-1.75,0)}] (-.1,-2)--+(-.6,0) (-.4,-2) node[particle] (v1) {$A^X$} (.1,-2)--+(.6,0) (.4,-2) node[particle] (v2) {$B$} node[particle,shift={(0,.35)}] (av1) {$B^Y$}            ;
        \draw[thick,->] (v1) to [out=45,in=135] (v2);
        \draw[thick,shift={(1.75,0)}] (-.1,-2)--+(-.6,0)  (.1,-2)--+(.6,0) (.4,-2) node[particle] (av2) {$B^X$}             node[particle,shift={(0,.35)}] (av1) {$B^Y$};
        \draw[decoration=snake,decorate,->] (-.75,-1.9)--(.75,-1.9);
        
        \begin{scope}[shift={(0.0,-1)}]
          \draw[thick,shift={(-1.75,0)}] (-.1,-2)--++(-.6,0) node[anchor=base east] {$\zeta$:}(-.4,-2) node[particle] (v1) {$\phantom{A}$} (.1,-2)--+(.6,0) (.4,-2) node[particle] (v2) {$B$}   node[particle,shift={(0,.35)}] (av1) {$B$}          ;
        \draw[thick,shift={(1.75,0)}] (-.1,-2)--+(-.6,0)  (.1,-2)--+(.6,0) (.4,-2) node[particle] (av2) {$B$}
            node[particle,shift={(0,.35)}] (av1) {$B$};
        \draw[decoration=snake,decorate,->] (-.75,-1.9)--(.75,-1.9);

        \begin{scope}[shift={(0.0,-1)}]
          \draw[thick,shift={(-1.75,0)}] (-.1,-2)--++(-.6,0) node[anchor=base east] {$\zeta^X$:}(-.4,-2) node[particle] (v1) {$A$} (.1,-2)--+(.6,0) (.4,-2) node[particle] (v2) {$B$}     node[particle,shift={(0,.35)}] (av1) {$B$}         ;
        \draw[thick,->] (v1) to [out=45,in=135] (v2);
        \draw[thick,shift={(1.75,0)}] (-.1,-2)--+(-.6,0)  (.1,-2)--+(.6,0) (.4,-2) node[particle] (av2) {$B$}             ;
        \draw[decoration=snake,decorate,->] (-.75,-1.9)--(.75,-1.9);

        \begin{scope}[shift={(0,-1)}]
          \draw[thick,shift={(-1.75,0)}] (-.1,-2)--++(-.6,0) node[anchor=base east] {$\zeta^{X,Y}$:}(-.4,-2) node[particle] (v1) {$A$} (.1,-2)--+(.6,0) (.4,-2) node[particle] (v2) {$B$}             ;
        \draw[thick,->] (v1) to [out=45,in=135] (v2);
        \draw[thick,shift={(1.75,0)}] (-.1,-2)--+(-.6,0)  (.1,-2)--+(.6,0) (.4,-2) node[particle] (av2) {$\phantom{A}$}             ;
        \draw[decoration=snake,decorate,->] (-.75,-1.9)--(.75,-1.9);
        \end{scope}
        \end{scope}
        \end{scope}
        \end{scope}

      \end{scope}

      \begin{scope}[shift={(3,-6.5)}]
        \path (-3.6,0) node[text width=6.25cm,anchor=base west] {Case 3: The $X$-tracer in state~$A$ jumps onto a site containing the $Y$-tracer in state~$B$ and no other particles. The tracers swap states.};
              
        \begin{scope}[shift={(0,-.2)}]
          \draw[thick,shift={(-1.75,0)}] (-.1,-2)--+(-.6,0) (-.4,-2) node[particle] (v1) {$A^X$} (.1,-2)--+(.6,0) (.4,-2) node[particle] (v2) {$B^Y$}         ;
        \draw[thick,->] (v1) to [out=45,in=135] (v2);
        \draw[thick,shift={(1.75,0)}] (-.1,-2)--+(-.6,0)  (.1,-2)--+(.6,0) (.4,-2) node[particle] (av2) {$A^Y$}         node[particle,shift={(0,.35)}] (av1) {$B^X$}            ;
        \draw[decoration=snake,decorate,->] (-.75,-1.9)--(.75,-1.9);
        
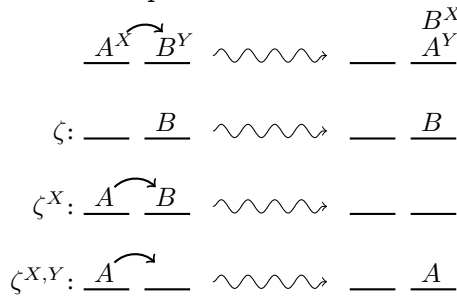
        \begin{scope}[shift={(0.0,-1)}]
          \draw[thick,shift={(-1.75,0)}] (-.1,-2)--++(-.6,0) node[anchor=base east] {$\zeta$:}(-.4,-2) node[particle] (v1) {$\phantom{A}$} (.1,-2)--+(.6,0) (.4,-2) node[particle] (v2) {$B$}   ;
        \draw[thick,shift={(1.75,0)}] (-.1,-2)--+(-.6,0)  (.1,-2)--+(.6,0) (.4,-2) node[particle] (av2) {$B$};
        \draw[decoration=snake,decorate,->] (-.75,-1.9)--(.75,-1.9);

        \begin{scope}[shift={(0.0,-1)}]
          \draw[thick,shift={(-1.75,0)}] (-.1,-2)--++(-.6,0) node[anchor=base east] {$\zeta^X$:}(-.4,-2) node[particle] (v1) {$A$} (.1,-2)--+(.6,0) (.4,-2) node[particle] (v2) {$B$}             ;
        \draw[thick,->] (v1) to [out=45,in=135] (v2);
        \draw[thick,shift={(1.75,0)}] (-.1,-2)--+(-.6,0)  (.1,-2)--+(.6,0)             ;
        \draw[decoration=snake,decorate,->] (-.75,-1.9)--(.75,-1.9);

        \begin{scope}[shift={(0,-1)}]
          \draw[thick,shift={(-1.75,0)}] (-.1,-2)--++(-.6,0) node[anchor=base east] {$\zeta^{X,Y}$:}(-.4,-2) node[particle] (v1) {$A$} (.1,-2)--+(.6,0) (.4,-2) node[particle] (v2) {$\phantom{A}$}             ;
        \draw[thick,->] (v1) to [out=45,in=135] (v2);
        \draw[thick,shift={(1.75,0)}] (-.1,-2)--+(-.6,0)  (.1,-2)--+(.6,0) (.4,-2) node[particle] (v2) {$A$} ;
        \draw[decoration=snake,decorate,->] (-.75,-1.9)--(.75,-1.9);
        \end{scope}
        \end{scope}
        \end{scope}
        \end{scope}

      \end{scope}
      \begin{scope}[shift={(9.5,-6.5)}]
        \path (-3.6,0) node[text width=6.25cm,anchor=base west] 
           {Case 4: The $Y$-tracer in state~$A$ jumps onto a site containing no nontracer $B$-particles. 
           No interaction occurs.};
        \begin{scope}[shift={(0,-.2)}]
        \draw[thick,shift={(-1.75,0)}] (-.1,-2)--+(-.6,0) (-.4,-2) node[particle] (v1) {$A^Y$} (.1,-2)--+(.6,0) (.4,-2) node[particle] (v2) {$B^X$}             ;
        \draw[thick,->] (v1) to [out=45,in=135] (v2);
        \draw[thick,shift={(1.75,0)}] (-.1,-2)--+(-.6,0)  (.1,-2)--+(.6,0) (.4,-2) node[particle] (av2) {$B^X$}         node[particle,shift={(0,.35)}] (av1) {$A^Y$}    ;
        \draw[decoration=snake,decorate,->] (-.75,-1.9)--(.75,-1.9);
        
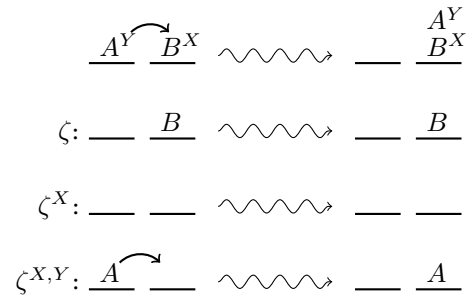
        \begin{scope}[shift={(0.0,-1)}]
          \draw[thick,shift={(-1.75,0)}] (-.1,-2)--++(-.6,0) node[anchor=base east] {$\zeta$:}(-.4,-2) node[particle] (v1) {$\phantom{A}$} (.1,-2)--+(.6,0) (.4,-2) node[particle] (v2) {$B$}             ;
        \draw[thick,shift={(1.75,0)}] (-.1,-2)--+(-.6,0)  (.1,-2)--+(.6,0) (.4,-2) node[particle] (av2) {$B$};
        \draw[decoration=snake,decorate,->] (-.75,-1.9)--(.75,-1.9);

        \begin{scope}[shift={(0.0,-1)}]
          \draw[thick,shift={(-1.75,0)}] (-.1,-2)--++(-.6,0) node[anchor=base east] {$\zeta^X$:}(-.4,-2) node[particle] (v1) {$\phantom{A}$} (.1,-2)--+(.6,0) (.4,-2) node[particle] (v2) {$\phantom{A}$}             ;
        \draw[thick,shift={(1.75,0)}] (-.1,-2)--+(-.6,0)  (.1,-2)--+(.6,0)            ;
        \draw[decoration=snake,decorate,->] (-.75,-1.9)--(.75,-1.9);

        \begin{scope}[shift={(0,-1)}]
          \draw[thick,shift={(-1.75,0)}] (-.1,-2)--++(-.6,0) node[anchor=base east] {$\zeta^{X,Y}$:}(-.4,-2) node[particle] (v1) {$A$} (.1,-2)--+(.6,0) (.4,-2) node[particle] (v2) {$\phantom{A}$}             ;
        \draw[thick,->] (v1) to [out=45,in=135] (v2);
        \draw[thick,shift={(1.75,0)}] (-.1,-2)--+(-.6,0)  (.1,-2)--+(.6,0) (.4,-2) node[particle] (av2) {$A$}           ;
        \draw[decoration=snake,decorate,->] (-.75,-1.9)--(.75,-1.9);
        \end{scope}
        \end{scope}
        \end{scope}
        \end{scope}

      \end{scope}
      \begin{scope}[shift={(6.25,-13)}]
        \path (-4.1,0) node[text width=7.25cm,anchor=base west] 
           {Case 5: The $Y$-tracer in state~$A$ jumps onto a site containing nontracer $B$-particles.
           It annihilates ones of them and then enters state~$B$.};
        \begin{scope}[shift={(0,0.2)}]
        \draw[thick,shift={(-1.75,0)}] (-.1,-2)--+(-.6,0) (-.4,-2) node[particle] (v1) {$A^Y$} (.1,-2)--+(.6,0) (.4,-2) node[particle] (v2) {$B$}     node[particle,shift={(0,.35)}] (av1) {$B^X$}        ;
        \draw[thick,->] (v1) to [out=45,in=135] (v2);
        \draw[thick,shift={(1.75,0)}] (-.1,-2)--+(-.6,0)  (.1,-2)--+(.6,0) (.4,-2) node[particle] (av2) {$B^Y$}         node[particle,shift={(0,.35)}] (av1) {$B^X$}    ;
        \draw[decoration=snake,decorate,->] (-.75,-1.9)--(.75,-1.9);
        
        \begin{scope}[shift={(0.0,-1)}]
          \draw[thick,shift={(-1.75,0)}] (-.1,-2)--++(-.6,0) node[anchor=base east] {$\zeta$:}(-.4,-2) node[particle] (v1) {$\phantom{A}$} (.1,-2)--+(.6,0) (.4,-2) node[particle] (v2) {$B$}   node[particle,shift={(0,.35)}] (av1) {$B$}              ;
        \draw[thick,shift={(1.75,0)}] (-.1,-2)--+(-.6,0)  (.1,-2)--+(.6,0) (.4,-2) node[particle] (av2) {$B$}
          node[particle,shift={(0,.35)}] (av1) {$B$};
        \draw[decoration=snake,decorate,->] (-.75,-1.9)--(.75,-1.9);

        \begin{scope}[shift={(0.0,-1)}]
          \draw[thick,shift={(-1.75,0)}] (-.1,-2)--++(-.6,0) node[anchor=base east] {$\zeta^X$:}(-.4,-2) node[particle] (v1) {$\phantom{A}$} (.1,-2)--+(.6,0) (.4,-2) node[particle] (v2) {$B$}             ;
        \draw[thick,shift={(1.75,0)}] (-.1,-2)--+(-.6,0)  (.1,-2)--+(.6,0)  (.4,-2) node[particle] (av2) {$B$}          ;
        \draw[decoration=snake,decorate,->] (-.75,-1.9)--(.75,-1.9);

        \begin{scope}[shift={(0,-1)}]
          \draw[thick,shift={(-1.75,0)}] (-.1,-2)--++(-.6,0) node[anchor=base east] {$\zeta^{X,Y}$:}(-.4,-2) node[particle] (v1) {$A$} (.1,-2)--+(.6,0) (.4,-2) node[particle] (v2) {$B$}             ;
        \draw[thick,->] (v1) to [out=45,in=135] (v2);
        \draw[thick,shift={(1.75,0)}] (-.1,-2)--+(-.6,0)  (.1,-2)--+(.6,0)           ;
        \draw[decoration=snake,decorate,->] (-.75,-1.9)--(.75,-1.9);
        \end{scope}
        \end{scope}
        \end{scope}
        \end{scope}

      \end{scope}
    \end{tikzpicture}

      \caption{When a tracer particle jumps in the tracer system, there are five cases.
      The same notation is used here as in Figure~\ref{fig:nontracer.jump}.}\label{fig:tracer.jump}
  \end{figure}
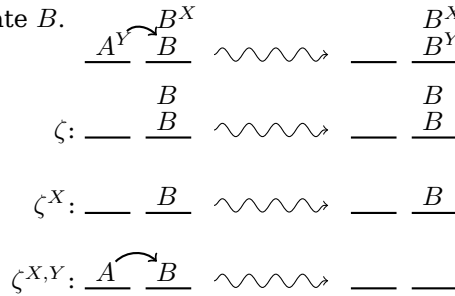

\FloatBarrier



\bibliographystyle{amsalpha}
\bibliography{icx}

\newcommand{\etalchar}[1]{$^{#1}$}
\providecommand{\bysame}{\leavevmode\hbox to3em{\hrulefill}\thinspace}
\providecommand{\MR}{\relax\ifhmode\unskip\space\fi MR }
\providecommand{\MRhref}[2]{%
  \href{http://www.ams.org/mathscinet-getitem?mr=#1}{#2}
}
\providecommand{\href}[2]{#2}
\begin{thebibliography}{TdALB{\etalchar{+}}12}

\bibitem[BBJ21]{bahl2019parking}
Riti Bahl, Philip Barnet, and Matthew Junge, \emph{Parking on supercritical
  {G}alton-{W}atson trees}, ALEA, Lat. Am. J. Probab. Math. Stat. \textbf{18}
  (2021).

\bibitem[BL88]{BL2}
Maury Bramson and Joel~L. Lebowitz, \emph{Asymptotic behavior of densities in
  diffusion-dominated annihilation reactions}, Phys. Rev. Lett. \textbf{61}
  (1988), 2397--2400.

\bibitem[BL90]{BL3}
Maury Bramson and Joel~L Lebowitz, \emph{Asymptotic behavior of densities in
  diffusion dominated two-particle reactions}, Physica A: Statistical Mechanics
  and its Applications \textbf{168} (1990), no.~1, 88--94.

\bibitem[BL91a]{BL4}
Maury Bramson and Joel~L. Lebowitz, \emph{Asymptotic behavior of densities for
  two-particle annihilating random walks}, J. Statist. Phys. \textbf{62}
  (1991), no.~1-2, 297--372. \MR{1105266}

\bibitem[BL91b]{BL5}
Maury Bramson and Joel~L Lebowitz, \emph{Spatial structure in diffusion-limited
  two-particle reactions}, Journal of statistical physics \textbf{65} (1991),
  no.~5-6, 941--951.

\bibitem[CEGM83]{collet1983study}
P~Collet, J-P Eckmann, V~Glaser, and A~Martin, \emph{Study of the iterations of
  a mapping associated to a spin glass model}, Les rencontres
  physiciens-math{\'e}maticiens de Strasbourg-RCP25 \textbf{33} (1983),
  117--142.

\bibitem[CG21]{chen2021}
Qizhao Chen and Christina Goldschmidt, \emph{Parking on a random rooted plane
  tree}, Bernoulli \textbf{27} (2021), no.~1, 93--106.

\bibitem[CH19]{curien2019phase}
Nicolas Curien and Olivier H{\'e}nard, \emph{The phase transition for parking
  on {G}alton-{W}atson trees}, arXiv:1912.06012 (2019).

\bibitem[CJJ{\etalchar{+}}21]{dlas_star}
Irina Cristali, Yufeng Jiang, Matthew Junge, Remy Kassem, David Sivakoff, and
  Grayson York, \emph{Two-type annihilating systems on the complete and star
  graph}, Stochastic Processes and their Applications \textbf{139} (2021),
  321--342.

\bibitem[Con20]{contat2020sharpness}
Alice Contat, \emph{Sharpness of the phase transition for parking on random
  trees}, Random Structures \& Algorithms (2020).

\bibitem[CRS18]{CRS}
M.~Cabezas, L.~T. Rolla, and V.~Sidoravicius, \emph{Recurrence and density
  decay for diffusion-limited annihilating systems}, Probab. Theory Related
  Fields \textbf{170} (2018), no.~3-4, 587--615. \MR{3773795}

\bibitem[DGJ{\etalchar{+}}19]{parking}
Michael Damron, Janko Gravner, Matthew Junge, Hanbaek Lyu, and David Sivakoff,
  \emph{Parking on transitive unimodular graphs}, The Annals of Applied
  Probability \textbf{29} (2019), no.~4, 2089--2113.

\bibitem[GD96]{gopich1996kinetics}
Irina~V Gopich and Alexander~B Doktorov, \emph{Kinetics of diffusion-influenced
  reversible reaction ${A} + {B} \to {C}$ in solutions}, The Journal of
  chemical physics \textbf{105} (1996), no.~6, 2320--2332.

\bibitem[GP19]{tree}
Christina Goldschmidt and Micha{\l} Przykucki, \emph{Parking on a random tree},
  Combinatorics, Probability and Computing \textbf{28} (2019), no.~1, 23--45.

\bibitem[Hil76]{hill1976homogeneous}
James~C Hill, \emph{Homogeneous turbulent mixing with chemical reaction},
  Annual review of fluid Mechanics \textbf{8} (1976), no.~1, 135--161.

\bibitem[Hut22]{hutchcroft2020transience}
Tom Hutchcroft, \emph{Transience and recurrence of sets for branching random
  walk via non-standard stochastic orders}, Annales de l'Institut Henri
  Poincar{\'e}, Probabilit{\'e}s et Statistiques, vol.~58, Institut Henri
  Poincar{\'e}, 2022, pp.~1041--1051.

\bibitem[JJ18]{JJ}
Tobias Johnson and Matthew Junge, \emph{Stochastic orders and the frog model},
  Ann. Inst. Henri Poincar\'e Probab. Stat. \textbf{54} (2018), no.~2,
  1013--1030. \MR{3795075}

\bibitem[JJLS20]{johnson2020particle}
Tobias Johnson, Matthew Junge, Hanbaek Lyu, and David Sivakoff, \emph{Particle
  density in diffusion-limited annihilating systems}, arXiv:2005.06018 (2020).

\bibitem[JLS12]{jerison2012logarithmic}
David Jerison, Lionel Levine, and Scott Sheffield, \emph{Logarithmic
  fluctuations for internal dla}, Journal of the American Mathematical Society
  \textbf{25} (2012), no.~1, 271--301.

\bibitem[JR19]{johnson2019}
Tobias Johnson and Leonardo~T. Rolla, \emph{Sensitivity of the frog model to
  initial conditions}, Electron. Commun. Probab. \textbf{24} (2019), 9 pp.

\bibitem[KW66]{KW}
Alan~G. Konheim and Benjamin Weiss, \emph{An occupancy discipline and
  applications}, SIAM Journal on Applied Mathematics \textbf{14} (1966), no.~6,
  1266--1274.

\bibitem[LBG92]{lawler1992internal}
Gregory~F Lawler, Maury Bramson, and David Griffeath, \emph{Internal diffusion
  limited aggregation}, The Annals of Probability (1992), 2117--2140.

\bibitem[Mar02]{Marchand}
R.~Marchand, \emph{Strict inequalities for the time constant in first passage
  percolation}, Ann. Appl. Probab. \textbf{12} (2002), no.~3, 1001--1038.
  \MR{1925450}

\bibitem[OZ78]{ovchinnikov1978role}
AA~Ovchinnikov and Ya~B Zeldovich, \emph{Role of density fluctuations in
  bimolecular reaction kinetics}, Chemical Physics \textbf{28} (1978), no.~1-2,
  215--218.

\bibitem[PRS19]{parking_on_integers}
Micha{\l} Przykucki, Alexander Roberts, and Alex Scott, \emph{Parking on the
  integers}, arXiv:1907.09437 (2019).

\bibitem[RK00]{raje2000experimental}
Deepashree~S Raje and Vivek Kapoor, \emph{Experimental study of bimolecular
  reaction kinetics in porous media}, Environmental science \& technology
  \textbf{34} (2000), no.~7, 1234--1239.

\bibitem[RSSS19]{rivera2019dispersion}
Nicol{\'a}s Rivera, Thomas Sauerwald, Alexandre Stauffer, and John Sylvester,
  \emph{The dispersion time of random walks on finite graphs}, The 31st ACM
  Symposium on Parallelism in Algorithms and Architectures, 2019, pp.~103--113.

\bibitem[SS07]{SS}
Moshe Shaked and J.~George Shanthikumar, \emph{Stochastic orders}, Springer
  Series in Statistics, Springer, New York, 2007. \MR{2265633 (2008g:60005)}

\bibitem[TdALB{\etalchar{+}}12]{tartakovsky2012effect}
Alexandre~M Tartakovsky, Pietro de~Anna, Tanguy Le~Borgne, A~Balter, and Diogo
  Bolster, \emph{Effect of spatial concentration fluctuations on effective
  kinetics in diffusion-reaction systems}, Water Resources Research \textbf{48}
  (2012), no.~2.

\bibitem[TW83]{toussaint1983particle}
Doug Toussaint and Frank Wilczek, \emph{Particle--antiparticle annihilation in
  diffusive motion}, The Journal of Chemical Physics \textbf{78} (1983), no.~5,
  2642--2647.

\bibitem[vdBK93]{BK}
J.~van~den Berg and H.~Kesten, \emph{Inequalities for the time constant in
  first-passage percolation}, Ann. Appl. Probab. \textbf{3} (1993), no.~1,
  56--80. \MR{1202515}

\end{thebibliography}




\end{document}